\documentclass[12pt]{amsart}
\usepackage{amssymb, a4wide, mathdots, url,hyperref,graphicx}
\usepackage[usenames]{color}
\usepackage{enumerate}
\usepackage{wasysym}
\usepackage{dsfont}
\usepackage{tikz-cd}
\usepackage{comment}
\usepackage{MnSymbol}
\usepackage{kotex}
\usepackage[shortlabels]{enumitem}
\newcommand{\abs}[1]{\left|{#1}\right|}

\newcommand{\sgn}{\mathrm{sgn}}
\newcommand{\z}{\overline{z}}
\newcommand{\Z}{\mathbb Z}
\newcommand{\Ha}{\mathbb H}
\newcommand{\K}{\mathbb K}
\newcommand{\C}{\mathbb C}
\newcommand{\R}{\mathbb R}

\newcommand{\cg}{\textnormal{\textsl{g}}}

\newcommand{\SL}{\operatorname{SL}}

\newcommand{\cL}{\mathcal{L}}
\newcommand{\idd}{\mathds{1}}
\newcommand{\sv}{s^{\vee}}

\newtheorem{theorem}{Theorem}[section]
\newtheorem{lemma}[theorem]{Lemma}

\newtheorem{definition}[theorem]{Definition}
\newtheorem{proposition}[theorem]{Proposition}%[subsection]
\newtheorem{corollary}[theorem]{Corollary}%[subsection]
\newtheorem{question}[theorem]{Question}%[subsection]
\theoremstyle{remark}
\newtheorem{remark}[theorem]{Remark}%[subsection]
\makeatletter
\@namedef{subjclassname@2020}{%
 \textup{2020} Mathematics Subject Classification}
\makeatother

\begin{document}

\title[Non-holomorphic Eisenstein series and the class numbers]{Non-holomorphic Eisenstein series for certain Fuchsian groups and class numbers}

\author{Bo-Hae Im}
\address{Department of Mathematical sciences,  KAIST, 291 Daehak-ro, Yuseong-gu, Daejeon 34141, South Korea}
\email{bhim@kaist.ac.kr}

\author{Wonwoong Lee}
\address{Department of Mathematical sciences,  KAIST, 291 Daehak-ro, Yuseong-gu, Daejeon 34141, South Korea}
\email{leeww@kaist.ac.kr}

\thanks{This work was supported by the National Research Foundation of Korea(NRF) grant funded by the Korea government(MSIT)
(No.~2020R1A2B5B01001835).}

\date{\today}

\begin{abstract}   
We study certain types of Fuchsian groups of the first kind denoted by $R(N)$, which  coincide with the Fricke groups or the arithmetic Hecke triangle groups of low levels. We find all elliptic points and cusps of $R(p)$ for a prime $p$, and prove that there is a one-to-one correspondence between the set of equivalence classes of elliptic points of $R(p)$ and the imaginary quadratic class group. We also find the explicit formula of the Fourier expansion of the non-holomorphic Eisenstein series for $R(N)$ and study their analytic properties. These non-holomorphic Eisenstein series together with cusp forms provide a basis for the space of polyharmonic Maass forms for $R(N)$.
\end{abstract}

\maketitle

%\goodbreak 

\setcounter{tocdepth}{2}
\tableofcontents

%\goodbreak

\section{Introduction} \label{intro}
Let $k$ be an even integer and $$\Delta_k:=y^2\left(\frac{\partial^2}{\partial x^2}+\frac{\partial^2}{\partial y^2}\right)-iky\left(\frac{\partial}{\partial x}+\frac{\partial}{\partial y}\right)$$ be the hyperbolic laplacian on the complex upper half plane $\mathbb H$. A shifted polyharmonic Maass form $f$ of weight $k$ for a Fuchsian group $\Gamma$ is a complex-valued smooth function on $\mathbb H$ satisfying the following conditions:
\begin{itemize}
    \item (modular invariant) $f|_k \gamma (z)=f(z)$ for any $\gamma \in \Gamma$, where $$ (f|_k \gamma) (z) :=(cz+d)^{-k}f\left(\frac{az+b}{cz+d}\right), \quad \gamma=\begin{pmatrix} a & b \\ c & d \end{pmatrix}. $$ We call $|_k$ a slash operator.
    \item (polyharmonic) $(\Delta_k-\lambda)^m f(z)=0$ for some $m \in \mathbb Z$ and $\lambda \in \mathbb C$.
    \item (moderate growth at cusps) $f$ grows polynomially at cusps in $y$, i.e. for each cusp $c \in \mathbb R$ of $\Gamma$ and $\sigma \in \text{SL}_2(\mathbb R)$ that sends the cusp $c$ to $\infty$, there exists $\alpha \in \mathbb R$ such that $(f|_k \sigma^{-1}) (x+iy)=O(y^{\alpha})$ as $y \to \infty$, uniformly in $x$.
\end{itemize}
A complex number $\lambda$ is called an eigenvalue of $f$ and a non-negative integer $m$ is called the (harmonic) depth of $f$.  One can define a shifted polyharmonic weak Maass form by replacing the moderate growth condition with the exponential growth.

The notion of shifted polyharmonic Maass forms includes the notion of classical Maass forms. When a depth $m$ is equal to $1$, it is a classical Maass cusp form (of weight $k$). Further if $\lambda=0$, it is called a harmonic Maass form. If we have $\lambda=0$ only, then we call it a polyharmonic Maass form. Obviously, this notion includes the notion of modular forms, which can be seen as holomorphic harmonic Maass forms.

Polyharmonicity is parallel with the literature on the polyharmonic functions for the Euclidean laplacian operator, which are classical and historic, extensive literature. See~\cite{MR0745128}, \cite{Emi99}, and \cite{MR2401623}. 

Polyharmonic Maass forms for the full modular group $\text{SL}_2(\mathbb Z)$ have been studied in recent years. Lagarias and Rhoades~\cite{MR357462} studied the space of polyharmonic Maass forms (with eigenvalue zero) $V_k^m(0)$. More precisely, they explicitly exhibit a basis of $V_k^m(0)$. 

One can define the notion of shifted polyharmonic Maass forms of the half-integer depth~$m+\frac{1}{2}$ by replacing the polyharmonic condition with
\begin{itemize}
    \item $(\Delta_k-\lambda)^m f(z)=g(z)$ for a holomorphic modular form $g(z)$.
\end{itemize}
Lagarias and Rhoades~\cite{MR357462} also determined the space of polyharmonic Maass forms of harmonic depth, completely.
In the sequel paper~\cite{MR3856183} of that, Andersen, Lagrias and Rhoades deal with the space of shifted polyharmonic Maass forms $V_k^m(\lambda)$ for nonzero eigenvalue $\lambda$. They showed that  the space $V_k^m(\lambda)$ is also finite-dimensional, found the upper bound of the dimension, decomposed the space into the direct sum of two subspaces, and exhibited a basis of one of them. 

Special kinds of polyharmonic Maass forms for congruence groups are also studied by several authors. Bruinier, Funke and Imamoḡlu~\cite{MR3353542} constructed the general regularized theta lift from the weak Maass forms of weight zero to the polyharmonic Maass forms of weight $1/2$. In particular, applying the lift to the constant function $1$, we get a certain polyharmonic Maass form of weight $1/2$ and depth $3/2$ for the congruence group $\Gamma_0(4)$, whose Fourier coefficients are related to real quadratic class numbers. Ahlgren, Andersen and Samart~\cite{MR3852563} introduced the analogous form for the full modular group $\text{SL}_2(\mathbb Z)$ and studied them. After then, the studies in a similar context on the space of polyharmonic weak Maass forms, which includes the notion of polyharmonic Maass forms, have also been actively conducted. See~\cite{MR3353542}, \cite{MR4198754}, \cite{MR3893030} and \cite{MR4047158}.

On the other hand, over the Hecke triangle groups and the Fricke groups, there have been a lot of developments about the theory of modular forms and their arithmetic properties, in each field, respectively. Let us introduce them briefly. For positive integers $m_1,m_2,m_3$ which are allowed to be infinity, with $1/m_1+1/m_2+1/m_3<1$,  a triangle group $\Gamma_{(m_1,m_2,m_3)}$ as an abstract group is defined by the group presentation
$$
\Gamma_{(m_1,m_2,m_3)}=\langle g_1, g_2 : g_1^{m_1}=g_2^{m_2}=(g_1 g_2)^{m_3}=1\rangle.
$$
In the presentation, we use the convention that the infinity-power is the identity. For example, $\Gamma_{(2,3,\infty)}=\{g_1, g_2 : g_1^2=g_2^3=1 \}$, which is isomorphic to $\text{PSL}_2(\mathbb Z)$. One can explicitly define $\Gamma_{(m_1,m_2,m_3)}$ as a subgroup of $\text{SL}_2(\mathbb R)$ as follows:
$$
\Gamma_{(m_1,m_2,m_3)}=\langle \gamma_1, \gamma_2, \gamma_3 \rangle,
$$
where $$\gamma_1=\begin{pmatrix} 2\cos(\pi/m_1) & 1 \\ -1 & 0 \end{pmatrix},\quad \gamma_2=\begin{pmatrix} 0 & 1 \\ -1 & 2\cos(\pi/m_2) \end{pmatrix},\quad \gamma_3=\begin{pmatrix} 1 & 2\cos(\pi/m_1)+2\cos(\pi/m_2)  \\ 0 & 1 \end{pmatrix}.$$ The other generators conjugate to these matrices can be chosen.

Doran, Gannon, Movasati and Shokri \cite{MR3228299} developed the theory of modular forms for triangle groups. They displayed the basis of modular forms on this group, and moreover explored the arithmeticity of the Fourier coefficients of them. Movasati and Shokri continue their study on the Fourier coefficients for non-arithmetic triangle groups in their sequel paper \cite{MR3253293}. These modular forms and related objects are studied from various perspectives in recent years. For instance, see \cite{MR4045722}, \cite{MR4160413} and \cite{MR3962509}.

The Fricke group of level $N$ is a subgroup of $\text{SL}_2{(\mathbb R)}$ generated by the congruence group~$\Gamma_0(N)$ and the matrix $\omega_N=\begin{pmatrix} 0 & -1/\sqrt N \\ \sqrt N & 0 \end{pmatrix}.$ It is a subgroup of the well-known group $\Gamma_0^+(N)$ which is generated by $\Gamma_0(N)$ and Atkin-Lehner involutions $\omega_e$ for the Hall divisors~$e$ of $N$. $\Gamma_0^+(N)$ is closely related to the congruence group $\Gamma_0(N)$ and the Hecke operators of level $N$ by the Atkin-Lehner theory. Hecke eigenforms are classified by the eigenvalues under the involution. Its eigenvalue is $1$ (in which case it is modular under~$\Gamma_0^+(N)$) or -1. In particular, if $e=N$, we call $w_e$ the Fricke involution.
The theory of Atkin-Lehner which was first developed by Atkin and Lehner \cite{MR0268123}, in particular, the theory on modular forms for the Fricke group is also far reached.

Two different types of the Fuchsian groups we introduced coincide with each other in the low level cases. When $N=2$ or $3$, the Atkin-Lehner group $\Gamma_0^+(N)$ is exactly the Fricke group of the same level, and they are conjugate to the Hecke triangle groups of type $(2,2N,\infty)$. They are triangle groups and the corresponding modular curves have genus~$0$. Furthermore, the Hecke triangle group $\Gamma_{(2,\infty,\infty)}$ is also a well-known group and it is conjugate to $\Gamma_0(2)$.

These groups $\Gamma_0^+(2)$ and $\Gamma_0^+(3)$ have been studied independently as the first two non-trivial simple examples of the Fricke groups or the Hecke triangle groups, from various points of view. For example, there have been results not only about the space of automorphic forms itself or its members, but also concerning the zeros of automorphic forms, linear relations on the Poincare series, or the period functions, and so on. For instance, see~\cite{MR3462568}, \cite{MR2344823}, or other subsequent results.

\

In this paper, we consider a new kind of the Fuchsian groups of the first kind which coincide with the Fricke groups or the Hecke triangle groups in the low level. We introduce them in the following subsection.

\subsection{Elliptic points and genera}\label{subsec_ell}

 Let $N \in \mathbb Z$ be a positive integer and let
\begin{align}\label{RN+-def}
    R(N)^+=\left\{\begin{pmatrix} a & b\sqrt N  \\ c\sqrt N & d \end{pmatrix} \in \text{SL}_2(\mathbb R) : a,b,c,d \in \mathbb Z \right\}, \\
    R(N)^-=\left\{\begin{pmatrix} a\sqrt N & b  \\ c & d\sqrt N \end{pmatrix} \in \text{SL}_2(\mathbb R) : a,b,c,d \in \mathbb Z \right\}.
\end{align}
Each of these sets is a not a group in general, but if we define
\begin{align}\label{Rndef}
    R(N)=R(N)^+ \cup R(N)^-,
\end{align}
then obviously $R(N)$ is a discrete subgroup of $\text{SL}_2(\mathbb R)$. Moreover, it turns out that they are finitely generated as abstract groups and in fact, are the Fuchsian groups with cofinite areas. Also, it turns out that they coincide with the Fricke groups or the Hecke triangle groups in the low level cases (see Section \ref{sec_RN}). 

Let $\text{Ell}^-(N)$ be the collection of $R(N)$-equivalence classes of elliptic points for $R(N)^-$, which means the points stabilized by some elliptic element in $R(N)^-$. Denote by $\text{Ell}^+(N)$  the set-theoretic complement of $\text{Ell}^-(N)$ in the collection of all equivalent classes of elliptic points for $R(N)$. We already know all elliptic points and cusps when $p=2,3$. For $p=2$, any elliptic point is equivalent to $i$ or $\frac{1+i}{\sqrt 2}$ and for $p=3$, is equivalent to $i$ or $\frac{\sqrt 3 +i}{2}.$ Any cusp for $R(p)$ when $p=2,3$ is equivalent to $\infty.$ Thus in most cases we assume $p \geq 5$.

We investigate where the elliptic points lie, what the cusps are, and how they are related to the class group of quadratic forms. Specifically, we prove our first main theorem:
\begin{theorem}\label{thm_ell}
Let $p \geq 5$ be a prime.
\begin{enumerate}[\normalfont(a)]
    \item $\#\left(\mathrm{Ell}^+(p)\right)=\begin{cases} 2 & \text { if } p \equiv 1\pmod{12}, \\ 1 & \text { if } p \equiv 5,7 \pmod{12}, \\ 0 & \text { if } p \equiv 11  \pmod{12}.  \end{cases}$
    \item There is a one-to-one correspondence between $\mathrm{Ell}^-(p)$ and the set of all reduced forms of (not necessarily primitive) positive definite binary quadratic forms of discriminant $-4p$.
\end{enumerate}
\end{theorem}
In particular, there is a correspondence between the set of elliptic points for $R(p)^-$ and the ideal class group of an imaginary quadratic field. Namely, it is the ideal class group $C(-4p)$ of an imaginary quadratic field of discriminant $-4p$ when $p \equiv 1 \pmod{4}$, or a disjoint union of the two ideal class groups $C(-4p)$ and $C(-p)$ of discriminant $-4p$ and $-p$, respectively, when $p \equiv 3 \pmod 4$. Its number of elements is obtained from the class number $h(d)$ of discriminant $d=-4p$ or $d=-p$.

There is a classical result to compute the dimension of the space of holomorphic modular forms for a Fuchsian group of the first kind from the information on the modular curves. For example, see \cite{MR1021004}. In particular, if we find the elliptic points, the cusps, their orders, and the areas of curves as a $2$-dimensional Riemannian manifold whose metric is induced by the standard hyperbolic metric on the upper half plane $\Ha$, then we can determine the genera of the curves $X(R(p))$ which are the compactified modular curves associated with the Fuchsian group $R(p)$. One can define the area of $X(R(p))$ by the area of its fundamental domain contained in $\Ha$, see \cite{MR1021004}.

Together with Theorem \ref{thm_ell}, the information on the curve $X(R(p))$ implies the following Corollary.
\begin{corollary}\label{cor_genusthm}
Let $p$ be a prime number and $v_p$ be the area of $X(R(p))$. Let $$
h_p := \begin{cases} \frac{13}{6}+\frac{1}{2}h(-4p), & \text{ if }p \equiv 1 \pmod{12}, \\
\frac{3}{2}+\frac{1}{2}h(-4p), & \text{ if }p \equiv 5 \pmod{12}, \\
\frac{5}{3}+\frac{1}{2}h(-4p)+\frac{1}{2}h(-p), & \text{ if }p \equiv 7 \pmod{12},  \\
1+\frac{1}{2}h(-4p)+\frac{1}{2}h(-p), & \text{ if }p \equiv 11 \pmod{12}.
\end{cases}
$$
If $v_p <2\pi h_p$, then the genus $g_p$ of $X(R(p))$ is zero and $v_p=2\pi (h_p-2)$.
\end{corollary}
% \begin{corollary}\label{cor_genusthm}
% Let $p$ be a prime number and $v_p$ be the area of $X(R(p))$. Let $$
% h_p := \begin{cases} \frac{8}{3}+\frac{1}{2}h(-4p), & \text{ if }p \equiv 1 \pmod{12}, \\
% 2+\frac{1}{2}h(-4p), & \text{ if }p \equiv 5 \pmod{12}, \\
% \frac{5}{3}+\frac{1}{2}h(-4p)+\frac{1}{2}h(-p), & \text{ if }p \equiv 7 \pmod{12},  
% 1+\frac{1}{2}h(-4p)+\frac{1}{2}h(-p), & \text{ if }p \equiv 11 \pmod{12}.
% \end{cases}
% $$
% If $v_p <2\pi h_p$, then the genus $g_p$ of $X(R(p))$ is zero and $v_p=2\pi (h_p-2)$.
% \end{corollary}

In Section \ref{sec_RN}, we find the orders of all elliptic points and cusps for $R(p)$. Corollary \ref{cor_genusthm} is obtained by applying the formula \cite[Theorem 2.4.3]{MR1021004} of the area of modular curve
\begin{align*}
    \frac{1}{2\pi}v_p=2g_p-2+\sum_{z \in X(R(p))}(1-e_z^{-1}).
\end{align*}

%%cusp 정리.

\subsection{Non-holomorphic Eisenstein series}
Recall that the non-holomorphic Eisenstein series for the full modular group $\text{SL}_2(\mathbb Z)$ is
$$
G_k(z,s):=\frac{1}{2}\sum_{(m,n) \in \mathbb Z^2 \setminus \{(0,0)\}} \frac{y^s}{|mz+n|^{2s}(mz+n)^k}, \text{ where } \Re(s)>1-\frac{k}{2},
$$
or
$$
g(z,\overline{z},\alpha,\beta):=\sum_{(m,n) \in \mathbb Z^2 \setminus \{(0,0)\}}(mz+n)^{-\alpha}(m\overline{z}+n)^{-\beta}, 
$$
where $\alpha, \beta \in \C$ such that $\Re(\alpha+\beta)>2$ and $\alpha-\beta \in 2\Z$. Note that $G_k(z,s)=\frac{1}{2}y^s g(z,\overline{z},s+k,s)$. The series $G_k(z,s)$ can be normalized as
\begin{align*}
    G_k(z,s)&=\frac{1}{2}\sum_{(m,n) \in \mathbb Z^2 \setminus \{(0,0)\}} \frac{y^s}{|mz+n|^{2s}(mz+n)^k} \\
    &=\frac{1}{2}\sum_{t=1}^{\infty}\sum_{\substack{(m,n) \in \mathbb Z^2 \setminus \{(0,0)\} \\ \gcd(m,n)=t}} \frac{y^s}{|mz+n|^{2s}(mz+n)^k} \\
    &=\frac{1}{2}\zeta(2s+k)\sum_{t=1}^{\infty}\sum_{\substack{(c,d) \in \mathbb Z^2 \setminus \{(0,0)\} \\ \gcd(c,d)=1}} \frac{y^s}{|cz+d|^{2s}(cz+d)^k} \\
    &=\zeta(2s+k)E_k(z,s),
\end{align*}
and  $E_k(z,s)$ in the above can be written as a Poincare series
$$
E_k(z,s)=\sum_{\gamma \in \mathrm{SL}_2(\mathbb Z)_{\infty} \backslash \mathrm{SL}_2(\mathbb Z)} \left(\Im^s |_k \gamma\right) (z), \text{ where } \Im^s(z):=y^s \text{ for } z=x+iy \text{ with }  x,y\in \R.
$$
For that reason and considering the constant term of its Fourier expansion, $E_k(z,s)$ is called a normalized (non-holomorphic) Eisenstein series.

It is well known that the non-holomorphic Eisenstein series for $\SL_2(\Z)$ has several analytic properties. The functions $G_k(z,s)$ and $g(z,\z,\alpha,\beta)$ are harmonic with respect to some elliptic operators, so in particular they are smooth. They are defined over $\Re(s)>1-\frac{k}{2}$ or $\Re(\alpha+\beta)>2$ at first, but have analytic continuation in $s$ or $q=\alpha+\beta$ on the whole plane $\C$. Also they are invariant under the slash operator for $\SL_2(\Z)$, and satisfy certain functional equations.

To explore the similar ones for the group $R(N)$, consider the following sets: 
\begin{align*}
    &(\Z \times \Z)_{N, L}:=\bigsqcup_{t \in \Z_{>0}}\{(m,n)\in t\Z \times t\Z : \mathrm{gcd}(Nm,n)=t\}, \\
    &(\Z \times \Z)_{N, R}:=\bigsqcup_{t \in \Z_{>0}}\{(m,n)\in t\Z \times t\Z : \mathrm{gcd}(m,Nn)=t\},
\end{align*}
These are subsets of $(\Z \times \Z) \setminus \{(0,0)\}$. When $N=1$, they are exactly $(\Z \times \Z) \setminus \{(0,0)\}$. We define the parts of non-holomorphic Eisenstein series for $R(N)$ by
\begin{align*}
   &G_{N,k,L}(z,s):=\frac{1}{2}\sum_{(m,n) \in (\Z \times \Z)_{N,L}} \frac{y^s}{|\sqrt N mz+n|^{2s}(\sqrt N mz+n)^k}, \text{ where } \Re(s)>1-\frac{k}{2}, \\
   &G_{N,k,R}(z,s):=\frac{1}{2}\sum_{(m,n) \in (\Z \times \Z)_{N,R}} \frac{y^s}{| mz+\sqrt N n|^{2s}(mz+\sqrt N n)^k}, \text{ where } \Re(s)>1-\frac{k}{2},
\end{align*}
and
\begin{align*}
    &g_{N,L}(z,\overline{z},\alpha,\beta):=\sum_{(m,n) \in (\Z \times \Z)_{N,L}}(\sqrt N mz+n)^{-\alpha}(\sqrt N m\z+n)^{-\beta}, \text{ where } \Re(\alpha+\beta)>2, \\
    &g_{N,R}(z,\overline{z},\alpha,\beta):=\sum_{(m,n) \in (\Z \times \Z)_{N,R}}(mz+\sqrt N n)^{-\alpha}(m\z+\sqrt N n)^{-\beta}, \text{ where } \Re(\alpha+\beta)>2.
\end{align*}
We define the non-holomorphic Eisenstein series for $R(N)$ as the sum of these two types of functions, $$G_{N,k}(z,s):=G_{N,k,L}(z,s)+G_{N,k,R}(z,s),$$ correspondingly $$g_N(z,\z,\alpha,\beta):=g_{N,L}(z,\z,\alpha,\beta)+g_{N,R}(z,\z,\alpha,\beta).$$ Note that $G_{1,k,L}(z,s)=G_{1,k,R}(z,s)=G_k(z,s)$,  $g_{1,L}(z,\z,\alpha,\beta)=g_{1,R}(z,\z,\alpha,\beta)=g(z,\z,\alpha,\beta)$ and $G_{N,k}(z,s)=\frac{1}{2}y^s g_N(z,\overline{z},s+k,s)$. If $N=1$, these functions satisfy the analytic property and the automorphy condition. 

Before stating Theorem~\ref{thm_gfourier} which provides the Fourier expansion of $G_{N,k}(z,s)$, we introduce several notations. Let
\begin{align}\label{eqn_Gausssumprin}
    G(b,\idd_{n}):=\sum_{a \in (\Z/n\Z)^{\times}}(\zeta_n^a)^b
\end{align}
be the Ramanujan sum, that is the sum of the $b$th powers of the $n$th primitive roots of unity. In particular, if $b=1$, then $G(1,\idd_n)=\mu(n)$, the Mobius function. We denote the Dirichlet series associated to this sum by
\begin{align*}
    \cL(b,\idd_{\bullet},s):=\sum_{n=1}^{\infty}\frac{G(b,\idd_n)}{n^s}.
\end{align*}
Similarly, we define the associated local factors,
\begin{align*}
    \cL_N(b,\idd_{\bullet},s):=\prod_{p|N}\sum_{n=0}^{\infty}\frac{G(b,\idd_{p^n})}{p^{ns}}
\end{align*}
and the associated shifted series
\begin{align*}
    \cL(b,\idd_{N\bullet},s):=\sum_{n=1}^{\infty}\frac{G(b,\idd_{Nn})}{n^s}.
\end{align*}
By definition, $\cL(0,\idd_{\bullet},s)$ and $\cL(1,\idd_{\bullet},s)$ are just the Dirichlet series associated to the Euler totient function and the Mobius function, respectively. It seems unnecessary at this moment to use the notation $\idd_{\bullet}$, but in Section~\ref{sec_Eisen} we generalize this notion for the other {\it dual characters} instead of $\idd_{\bullet}$, and then it would fit the current definitions. It coincides with the current definitions when the dual character is trivial. 

Now we are ready to introduce our second results providing the Fourier expansions of non-holomorphic Eisenstein series for $R(N).$

\begin{theorem}\label{thm_gfourier}
For $\alpha, \beta\in \C$ such that $q:=\alpha+\beta \in \C_{\Re>2}$ and $k:=\alpha-\beta \in 2\Z$, if $N>1$, the Fourier expansion of a non-holomorphic Eisenstein series $g_{N}(z,\z,\alpha,\beta)$ in $x$ is given by
\begin{align*}
    g_N(z,\z,\alpha,\beta)=\phi&_{N}(y,q,k) \\
    &+2(-1)^k\zeta(q)\left(\frac{2\pi}{\sqrt N}\right)^q\sum_{n \neq 0}\frac{|n|^{q-1}}{\Gamma(\frac{q}{2}+\sgn(n)k)}W\left(2\pi n \frac{y}{\sqrt N};\alpha,\beta\right) \\
    &\times \left(N^{-q/2}\cL(n,\idd_{N\bullet},q)+\frac{\cL(n,\idd_{\bullet},q)}{\cL_N(n,\idd_{\bullet},q)}\right)e^{2\pi i nx/\sqrt N},
\end{align*}
where 
\begin{align*}
    \phi&_{N}(y,q,k)\\
    &:=2+2(-1)^k N^{-q/2}\zeta(q)\left(\frac{2y}{\sqrt N}\right)^{1-q}\frac{2\pi\Gamma(q-1)}{\Gamma(\frac{q+k}{2})\Gamma(\frac{q-k}{2})}\left(N^{-q/2}\cL(0,\idd_{N\bullet},q)+\frac{\cL(0,\idd_{\bullet},q)}{\cL_N(0,\idd_{\bullet},q)}\right)
\end{align*}
and $W(n;\alpha,\beta)$ is the modified Whittaker $W$-function.
\end{theorem}
For the detailed description of the modified Whittaker $W$-function, see Section~\ref{sec_Eisen}.

\

The issue of the analytic continuation of $g_N$ or $G_{N,k}$ is a bit more delicate to handle than that of the case $N=1$. If $N>1$, the $L$-series appearing in the Fourier expansion of $g_N(z,\z,\alpha,\beta)$ is more complicated than one in the Fourier expansion of $g(z,\z,\alpha,\beta)$. It is an unresolved question that whether $g_N$ has the analytic continuation in $q$ to the whole plane $\C$ for any positive integer~$N$. However, it is possible to continue  the series analytically at least to the neighborhood of $q=0$, which admits a taylor expansion of $g_n(z,\z,\alpha,\beta)$ near $q=0$, or of $G_{N,k}(z,s)$ near $s=0$. See Section \ref{sec_Eisen}.

Let $N=N_1^2 N_2$, where $N_2$ be a square-free integer. By definition, it is obvious that the group $R(N)$ is contained in $R(N_2)$, so our main interest is focused on the case when $N$ is square-free. From now on, we assume that $N$ is square-free. We define a meromorphic function in $s$,
\begin{align*}
    f_{N,b}(s)&:=\frac{N^{-(s-1)-\frac{k}{2}}\prod_{p \mid N}\left(\left(\sum_{i=1}^{v_p(b)}(p-1)p^{i-1+i(2s+k-2)}\right)-p^{v_p(b)+(v_p(b)+1)(2s+k-2)}\right)+1}{\prod_{p \mid N}\zeta_p^{-1}(-2s-k+2)\zeta_p(-2s-k+1;v_p(b))},
\end{align*}
where $v_p$ is the $p$-valuation and $\zeta_p(s;v):=\sum_{n=0}^v p^{-ns}$, a truncated local zeta function. Note that when $v_p(b)=0$, i.e., $p \nmid b$, the summation part $\sum_{i=1}^{v_p(b)}(p-1)p^{i-1+v_p(b)(2s+k-2)}$ doesn't appear in $f_{N,b}(s)$. In particular, if $b=1$, $$f_N(s):=f_{N,1}(s)=\left((-1)^{\omega(N)}N^{(s-1)+\frac{k}{2}}+1\right)\prod_{p|N}(1-p^{2s+k-2})^{-1}.$$ Here $\omega(N)$ is the number of distinct prime factors of $N$. 

Define the following completed and the doubly-completed Eisenstein series for $R(N)$, respectively by
\begin{align*}
    \widehat{G}_{N,k}(z,s)&:=\left(\frac{\pi}{\sqrt N}\right)^{-s-\frac{k}{2}}\Gamma\left(s+\frac{k}{2}+\frac{|k|}{2}\right)f_N(s)G_{N,k}(z,s), \\
    \widetilde{G}_{N,k}(z,s)&:=\left(s+\frac{k}{2}\right)\left(s+\frac{k}{2}-1\right)\widehat{G}_{N,k}(z,s).
\end{align*}
% \ww{(}If $\omega=\omega(N)$ and $\delta=1$, then we denote ${}_{\omega,\delta}\widehat{G}_{N,k}(z,s)$ by $\widetilde{G}_{N,k}(z,s)$, which we will be discussed later. However, most properties except for the analyticity at $s=0$ for small weight $k=0$ or $2$ are satisfied for general $\omega$ and $\delta$, so we'll discuss in the general setting.\ww{
% )}\bh{??}\ww{omega, delta 필요없으면 바꾸기. tilde로 통일.}

Theorem~\ref{thm_gfourier} provides the Fourier expansion of $\widehat{G}_{N,k}$  via the relation between the functions $g_N(z,\z,\alpha,\beta)$ and $\widehat{G}_{N,k}(z,s)$.

\begin{corollary}\label{cor_Gfourier}
If $N>1$ is square-free, the Fourier expansion of a non-holomorphic Eisenstein series $\widehat{G}_{N,k}(z,s)$ is given by
\begin{align*}
   \widehat{G}_{N,k}(z,s)=& \frac{\Gamma(s+\frac{k}{2}+\frac{|k|}{2})}{\Gamma(s+\frac{k}{2})}\hat{\zeta}(2s+k)y^s\sqrt N^{s+\frac{k}{2}}f_N(s) \\&+(-1)^{k/2}\frac{\Gamma(s+\frac{k}{2})\Gamma(s+\frac{k}{2}+\frac{|k|}{2})}{\Gamma(s+k)\Gamma(s)}\hat{\zeta}(2-2k-2s)y^{1-k-s}f_N(s) \\
   &\qquad \times \left(\sqrt N^{1-s-\frac{k}{2}}\prod_{p|N}\frac{1-p^{-1}}{1-p^{-2s-k}}+\sqrt N^{-1+s+\frac{k}{2}}\prod_{p|N}\frac{1-p^{1-k-2s}}{1-p^{-2s-k}}\right) \\
   &+\sum_{n \neq 0}(-1)^{k/2}y^{-k/2}\left(\frac{\Gamma(s+\frac{k}{2}+\frac{|k|}{2})}{\Gamma(s+\frac{k}{2}+\sgn(n)\frac{k}{2})}|n|^{-s-\frac{k}{2}}\sigma_{2s+k-1}(n)W_{\sgn(n)\frac{k}{2},s+\frac{k-1}{2}}\left(4\pi |n| \frac{y}{\sqrt N}\right)\right. \\
   &\qquad \times \left.f_{N,n}(1-k-s)f_N(s)e^{2\pi i nx/\sqrt N}\right),
\end{align*}
where $\hat{\zeta}(s):=\pi^{-s/2}\Gamma\left(\frac{s}{2}\right)\zeta(s)$ is the completed Riemann zeta function.
\end{corollary}

As in the case for $\SL_2(\Z)$, if $N$ is a prime, then the completed Eisenstein series satisfies the following functional equation centered at line $s=\frac{1-k}{2}$.

\begin{theorem}\label{thm_funeq}
For a prime $p$, the completed and doubly-completed Eisenstein series satisfy the functional equation,
\begin{align*}
    \widehat{G}_{p,k}(z,1-k-s)=\widehat{G}_{p,k}(z,s) \quad \text{ and } \quad
    \widetilde{G}_{p,k}(z,1-k-s)=\widetilde{G}_{p,k}(z,s).
\end{align*}
\end{theorem}

\subsection{Roadmap}

In Section \ref{sec_RN} we prove that the group $R(N)$ is finitely generated. Most group-theoretical properties of $R(N)$, including Theorem~\ref{thm_ell} and the fact that $R(N)$ coincides with the Fricke group $\Gamma_0^+(N)$ in the low level $N=2,3$, are given in this section. We derive Corollary~\ref{cor_genusthm} using the general and classical theorem on the theory of Fuchsian groups and we exhibit a few examples with $g_p=0$.

Section \ref{sec_Eisen} is devoted to developing the theory of $L$-series for the Gauss sum of dual characters. This leads us to the analytic continuation of the non-holomorphic Eisenstein series. In this section we prove that functions $G_{k,N,L}$ and $G_{k,N,R}$ satisfy the modular invariant condition for $R(N)$, and prove Theorem~\ref{thm_gfourier} and Theorem~\ref{thm_funeq}.

In section \ref{sec_polyhar}, we investigate the structure of polyharmonic Maass forms for $R(p)$. More preciesely, we prove that any polyharmonic Maass forms can be written as the sum of the Taylor coefficients of the non-holomorphic Eisenstein series and cusp forms. Most argument used in this section is in parallel with the argument in \cite{MR357462} once if we define and construct several properties on the non-holomorphic Eisenstein series. Thus we state our results briefly, and in most cases the proofs in this section are is omitted. For a general vector-valued version of the theory, one may see \cite{MR2097357} which gives the systematic investigation on the harmonic weak Maass forms.

\

\section{Groups $R(p)$}\label{sec_RN}

Section \ref{sec_RN} is devoted to developing the fundamental properties of the groups $R(p)$ and prove Theorem \ref{thm_ell}.

\subsection{Generating sets of $R(p)$}
In this subsection we first show that the group $R(p)$ is finitely generated.

\begin{lemma}\label{lem_gen}
Let $p>2$ be a prime. The group $R(p)$ is generated by matrices $\{T_p, \omega_1\} \cup \{s_p(n) \in \SL_2(\R): -\frac{p+1}{2} \leq n \leq \frac{p+1}{2},\text{ } n \neq 0 \}$, where 
\begin{align*}
    T_p:=\begin{pmatrix}1 & \sqrt p \\ 0 & 1 \end{pmatrix}, \qquad
    \omega_1 :=\begin{pmatrix}0 & -1 \\ 1 & 0 \end{pmatrix}, \qquad
    s_p(n):=\begin{pmatrix} \frac{n \hat{n}+1}{p}\sqrt p & n \\ \hat{n} & \sqrt p \end{pmatrix}
\end{align*}
and $\hat{n}$ is an integer satisfying $n \hat n \equiv -1 \pmod{p}$. For $p=2$, the group $R(p)$ is generated by $\{T_2,\omega_1,s_2(\pm 1)\}$.
\end{lemma}

It is possible to choice several $\hat{n}$ for each $n$, but regardless of the choice of $\hat{n}$ it generates the same group $R(p)$.

\begin{proof}
Let $p>2$ and denote  by $F$ the group generated by $T_p, \omega_1$, and $s_p(n)$'s. Let 
\begin{align*}
    F_{col}^-:=\Big\{(a,c) \in \Z^2 : &\text{ }|c| \leq \frac{p}{2}|a|, \text{ } \gcd(pa,c)=1, \\
    & \left.\text{ and there exist $b, d \in \Z$ such that }\begin{pmatrix} a\sqrt p & b \\ c & d \sqrt p\end{pmatrix} \in F \right\},
    \end{align*}
    \begin{align*}
    F^-:=\left\{a \in \Z : (a,c) \in F_{col}^- \text{ for any } c \text{ with } |c| \leq \frac{p}{2}|a|, \text{ } \gcd(pa,c)=1\right\}.
\end{align*}
Note that if $a \in F^-$ then $-a \in F^-$, so we may assume that $a>0$. In order to prove the assertion, we prove only  that 
\begin{align}\label{claim_F^-}
    \text{if any integer $m$ with $|m|<a$ belongs to $F^-$, then so is $a$.}
\end{align}
Indeed, this claim means that $\Z \subseteq F^-$. Note that the group $F$ is a subgroup of $\SL_2(\R)$. Thus if $a,b,c$ are given and $\begin{pmatrix} a\sqrt p & b \\ c & d \sqrt p\end{pmatrix} \in F$, an integer $d$ is uniquely determined. Suppose $(a_0,c_0) \in F_{col}^-$. Pick any integer $b_0$ such that $\begin{pmatrix} a_0\sqrt p & b_0 \\ c_0 & d_0 \sqrt p\end{pmatrix} \in F$, where $d_0$ is an integer determined by $a_0,c_0$ and $b_0$. Let $\begin{pmatrix} a_0\sqrt p & b \\ c_0 & d \sqrt p\end{pmatrix}$ be an arbitrary one belongs to $F$ with given the first column $(a_0\sqrt p,c_0)^t$. Then $pa_0d-bc_0=1$, so $b \equiv b_0 \pmod{pa_0}$. Write $b=pam+b_0$ for an integer $m$, then
\begin{align*}
    \begin{pmatrix} a_0\sqrt p & b \\ c_0 & d \sqrt p\end{pmatrix}=\begin{pmatrix} a\sqrt p & b_0 \\ c & d_0 \sqrt p\end{pmatrix}T_p^m.
\end{align*}
Hence $\Z = F^-$ implies $R(p)^- \subseteq F$. The set $R(p)^+$ can be written as $R(p)^+=\omega_1 R(p)^-$, so it follows that $R(p) \subseteq F$.

Let's prove Claim \eqref{claim_F^-}. Suppose that any integer with absolute value less than $a$ belongs to $F^-$. To prove $a\in F^-$, we show that $(a,c) \in F_{col}^-$ for any $c$ such that $\gcd(a,c)=1$ and $|c| \leq \frac{p}{2}a$. Use the notation $\hat{c}$ to mean that it satisfies $c\hat c \equiv -1 \pmod{p}$ and $|\hat{c}| \leq \frac{p}{2}a$. We say that $(a,c) \in \Z^2$ is  {\it properly  reduced} if there exists an integer $m \in \Z$ such that $|a+\hat{c}m|<a$, or equivalently, a matrix
\begin{align*}
    \begin{pmatrix} a\sqrt p & \hat{c} \\ c & d \sqrt p\end{pmatrix}\begin{pmatrix} \sqrt p & 1 \\ 1 & 0 \end{pmatrix}^m=\begin{pmatrix} a'\sqrt p & \hat{c} \\ c' & d \sqrt p\end{pmatrix}
\end{align*}
satisfies $|a'|<a$. Since $\begin{pmatrix} a'\sqrt p & \hat{c} \\ c' & d \sqrt p\end{pmatrix} \in F$, it is obvious that if $(a,c)$ is properly reduced, then $(a,c) \in F_{col}^-$. Therefore it is enough to show that any non-properly reduced pair $(a,c)$ is also contained in $F_{col}^-$.

Let $(a,c)$ be a non-properly reduced pair of integers with $|c| \leq \frac{p}{2}a$, $\gcd(pa,c)=1$. Since the inequality $|\hat{c}| < 2a$ implies $|a-\sgn(\hat{c})\hat c | <a$, we have $2a \leq |\hat{c}| \leq \frac{p}{2}a$, so there is an integer $n \in \{1, 2, \cdots, \frac{p+1}{2}\}$ with $(n-1)a < |\hat{c}| < (n+1)a.$ For such an integer $n$, we have $|an+|\hat{c}||=|\sgn(\hat{c})an+\hat{c}|<a$ and
\begin{align*}
    \begin{pmatrix} a\sqrt p & \hat{c} \\ c & d \sqrt p\end{pmatrix}s_p(\sgn(\hat{c})n)=\begin{pmatrix} * & (\sgn(\hat{n})an+\hat{c})\sqrt p \\ * & *\end{pmatrix}=A.
\end{align*}
By inductive hypothesis we have $A\omega_1 \in F$, so $\begin{pmatrix} a\sqrt p & \hat{c} \\ c & d \sqrt p\end{pmatrix}=A s_p(\sgn(\hat{c})n)^{-1} \in F$ which implies $(a,c) \in F_{col}^-$.

One can obtain the lemma for $p=2$ through the same  argument as above.
\end{proof}

By Lemma \ref{lem_gen}, it is clear that the Hecke triangle groups $\Gamma_{(2,4,\infty)}$ and $\Gamma_{(2,6,\infty)}$ coincide with groups $R(2)$ and  $ R(3)$, respectively.

\subsection{Elliptic points and cusps}

In this subsection we prove Theorem \ref{thm_ell}. Recall that a point $z \in \Ha $ is called an elliptic point for a Fuchsian group $\Gamma$ if its stabilizer group $\Gamma_z$ of the action $\Gamma$ on $\Ha$ has a nontrivial element. A nontrivial  element in $\Gamma_z$ always has trace whose absolute value is smaller than 2, and it is called an elliptic element. If two points $z_1, z_2 \in \Ha$ are $\Gamma$-equivalent, i.e. there exists a matrix $\gamma \in \Gamma$ such that $\gamma z_1=z_2$, then their stabilizer groups are conjugate to each other. Thus we also use the terminology `elliptic point' by meaning the equivalence class containing that point. 

For the Hecke triangle groups $\Gamma_{(2,m,\infty)}$, in particular for $m=4,6$, the set of equivalence classes of the elliptic points are known to be exactly $\{i, \frac{1+i}{\sqrt 2}\}$ and $\{i,\frac{\sqrt 3 +i}{2}\}$, respectively (see, for instance, \cite{MR3228299}). Therefore, unless specifically stated, we assume $p\geq 5$.

\begin{lemma}\label{lem_Fermat}
Let $c$ be a positive integer whose prime factors are $2$ or of the form $4r+1$ for some $r \in \Z$. Let $a$ and $a_0$ be integers satisfying
\begin{align*}
    \begin{cases}
a \equiv a_0 & \pmod{p}, \\
a^2 \equiv -1 & \pmod{p}, \\
\frac{a^2+1}{p} \equiv 0 & \pmod{c}.
\end{cases}
\end{align*}
There exists $(x,y) \in \Z^2$ such that
\begin{enumerate}[\normalfont(a)]
    \item $c=x^2+y^2$, \text{ and }
    \item $c$ divides $x-ay$ and $\frac{a-a_0}{p}x+\frac{a_0a+1}{p}y$.
\end{enumerate}
\end{lemma}

\begin{proof}
Let $a=a_0+np$, for $n \in \Z$. Then for any integers $x$ and $y$,
\begin{align*}
    \frac{a-a_0}{p}x+\frac{a_0a+1}{p}y=nx-nay+\frac{a^2+1}{p}y.
\end{align*}
Since  $p \mid a^2+1$ and $c \mid \frac{a^2+1}{p}$,  it is enough to show that there exists $(x,y) \in \Z^2$ such that
\begin{align*}
    c=x^2+y^2 \text{ and } c \mid (x-ay).
\end{align*}
Suppose $c$ is not divisible by 2. Note that we can replace $a$ with $a':=a+rpc$ for any $r \in \Z.$ Since $c$ is odd, we may assume that $a$ is even number. By considering $a$ as a Gaussian integer i.e., $a \in \Z[i]$, we have
\begin{align}\label{eq:gcd1}
    \gcd_{\Z[i]}(a+i,a-i)=\gcd_{\Z[i]}(a+i,2i)=\gcd_{\Z[i]}(a+i,2)=1,
\end{align}
so $a+i$ and $a-i$ are coprime in $\Z [i]$.

Note that $c$ can be factorized into
\begin{align*}
    c=\prod_{i=1}^n z_i^{e_i}\overline{z_i}^{e_i}
\end{align*}
in $\Z [i]$, where each $z_i$ is a prime in $\Z [i]$ and $\overline{z_i}$  is its complex conjugate, which is also a prime. Since $c$ divides $a^2+1=(a+i)(a-i)$, we may assume that $z_i \mid (a+i)$ for all $i=1,2,\ldots,n$, hence  $z_i^{e_i} \mid (a+i)$ by \eqref{eq:gcd1}. Let $x:=\Re(\prod_{i=1}^n z_i^{e_i})$, $y:=\Im(\prod_{i=1}^n z_i^{e_i}) \in \Z$. Then we have $c=x^2+y^2$ and
\begin{align*}
    c \mid (a+i)(x-iy)=(ax+y)+i(x-ay).
\end{align*}
Since $c$ is an integer, we have $c \mid (x-ay)$.

Next suppose $c$ is even. Since $a^2+1 \equiv 0 \pmod{c}$, we have $4 \nmid c$. If $a$ is even, then we can apply the same argument as in the case when $2 \mid c$. If  $a$ is an odd integer, then 
\begin{align*}
    \gcd_{\Z[i]}(a+i,a-i)=\gcd_{\Z[i]}(a+i,2)=\gcd_{\Z[i]}(1+i,2)=1+i,
\end{align*}
so we can write $a+i=(1+i)\prod_{i=1}^n w_i^{f_i}$ for primes $w_i \in \Z[i].$ Write $c=(1+i)(1-i)z_c \overline{z_c}$ for $z_c \in \Z[i]$. By the same reason as above, we may assume $z_c \mid \prod_{i=1}^n w_i^{f_i}.$ Define $x,y \in \Z$ so that $x+iy=(1+i)z_c$, then it satisfies the assertion.
\end{proof}

\begin{lemma}\label{lem_Fermat2}
Let $c$ be a positive integer whose prime factors are $3$ or of the form $3r+1$ for some $r \in \Z$. Let $a$ be an integer satisfying
\begin{align*}
a^2-a+1 \equiv 0 \pmod{c}.
\end{align*}
There exists $(x,y) \in \Z^2$ such that
\begin{enumerate}[\normalfont(a)]
    \item $c=3x^2+y^2$, \text{ and }
    \item c divides $(2a-1)x+y$.
\end{enumerate}
\end{lemma}

\begin{proof}
The proof is the same as the last part of the proof of Lemma \ref{lem_Fermat} except for using the ring of integers $\Z[\zeta_3]$ of $\mathbb{Q}(\zeta_3)$ instead of the Gaussian integers $\Z[i]$, where $\zeta_3:=\frac{-1+\sqrt{-3}}{2}$. Note that $\Z[\zeta_3]$ is a UFD and $a^2-a+1=(a+\zeta_3)(a+\overline{\zeta_3})$ in $\Z[\zeta_3]$.
\end{proof}

Consider an elliptic point in $\mathrm{Ell}(p)$ that is represented by $z \in \Ha.$  For another point $z' \in \Ha$ which is $R(p)$-equivalent to $z$, the trace of its associated elliptic element is equal to the trace of $\gamma,$ where $\gamma \in R(p)_z.$ Indeed, if $\alpha \in R(p)$ sends $z$ to $z'$, then any element of $R(p)_{z'}$ can be written as $\alpha \gamma \alpha^{-1}$.

Let $\mathrm{Ell}^+(p)_t$ be the collection of elliptic points $[z] \in \mathrm{Ell}^+(p)$ such that there is an associated elliptic element $\gamma \in R(p)_z$ whose trace is $t$. Since $|t| < 2$ and $-I$ acts trivially on $\Ha$, for any $t \in \Z$ we have $\mathrm{Ell}(p)_t=\mathrm{Ell}(p)_{-t}$, and the set of elliptic points $\mathrm{Ell}^+(p)$ can be written as a  union
\begin{align}\label{eqn_ellunion}
    \mathrm{Ell}^+(p)=\mathrm{Ell}^+(p)_0 \cup \mathrm{Ell}^+(p)_{1}.
\end{align}

In fact, it turns out that there is no elliptic point whose trace is $0$ and $1$ at the same time, i.e., the union of (\ref{eqn_ellunion}) is disjoint. This is clear by the following two propositions.

\begin{proposition}\label{prop_Ell+tr0}
Let $p$ be a prime. If $p=2$ or $p=4r+1$ for some $r \in \Z$, then the set of elliptic points $\mathrm{Ell}^+(p)_0$ is a singleton. Otherwise, $\mathrm{Ell}^+(p)_0= \emptyset$.
\end{proposition}

\begin{proof}
Let $z \in \mathrm{Ell}^+(p)_0$ be an elliptic point with trace $0$ and let $\gamma = \begin{pmatrix}
a & b \sqrt p \\ c \sqrt p & -a
\end{pmatrix} \in R(p)^+$ be one of its associated elliptic elements. Then $z$ can be written as $z=\frac{a+i}{c\sqrt p}.$ Note that $\det{\gamma}=-a^2-pbc=1$ implies $p \mid (a^2+1)$. This shows that if such $z \in \Ha$ exists, then $p$ must be 2 or of the form $4r+1$ for an integer $r \in \Z$. Moreover, we have $c \mid (a^2+1)$, so any prime factor $q$ of $c$ is 2 or satisfies $q \equiv 1 \pmod{4}$. 

For the rest of the proof, we suppose that $p=2$ or  $p=4r+1$ for some $r \in \Z$. Fix an integer $a_0 \in \Z$ such that $a_0^2+1 \equiv 0 \pmod{p}.$ Let $b_0:=\frac{a_0^2+1}{p}$ and $z_0:=\frac{a_0+i}{\sqrt p}$. It is an elliptic point of the element $\begin{pmatrix}
a_0 & b_0\sqrt p \\ \sqrt p & -a_0
\end{pmatrix}$, so $[z_0] \in \mathrm{Ell}^+(p)_0.$ Note that $[-\overline{z_0}] \in \mathrm{Ell}^+(p)_0$, which is inherited from $-a_0$ instead of $a_0$. We claim that these two elliptic points are $R(p)^-$-equivalent. Indeed, it is enough to show that there exist integers $X,Y,Z,W \in \Z$ such that
\begin{align*}
    \begin{pmatrix}
    X\sqrt p & Y \\ Z & W\sqrt p
    \end{pmatrix} \left(\frac{a_0+i}{\sqrt p}\right)=\frac{-a_0+i}{\sqrt p}.
\end{align*}
This can be reduced to the following system of equations by direct calculation,  
\begin{align*}
    \begin{cases}
\frac{a_0^2+1}{p}Z^2+2aZW+pW^2=1, \\
X=W, \\
Y=\frac{-a^2-1}{p}Z-2aW.
\end{cases}
\end{align*}
Notice that a quadratic form $\frac{a_0^2+1}{p}Z^2+2aZW+pW^2$ is positive definite of discriminant $-4$. Its reduced form is the principal form $Z^2+W^2,$ so it must represent~$1$. Thus the equation $\frac{a_0^2+1}{p}Z^2+2aZW+pW^2=1$ must has an integer solution $(Z,W) \in \Z^2.$ 

Next, we prove that any elliptic point $z \in \mathrm{Ell}^+(p)_0$ is $R(p)$-equivalent to $z_0$. Since $a^2+1 \equiv 0~\pmod{p}$, we may assume that $a \equiv a_0~\pmod{p}.$ By Lemma \ref{lem_Fermat}, there exists a pair of integers $(x_0,y_0) \in \Z^2$ such that
\begin{enumerate}[(a)]
    \item $c=x_0^2+y_0^2$, \text{ and }
    \item $c$ divides $x_0-ay_0$ and $\frac{a-a_0}{p}x_0+\frac{a_0a+1}{p}y_0$.
\end{enumerate}
Let \begin{align*}
    Z&:=y_0, \quad W:=-x_0-a_0y_0, \\ 
    X&:=\frac{(a_0+a)Z+W}{c}, \quad Y:=\frac{1}{c}\left(-\frac{a_0^2+1}{p}Z+\frac{a-a_0}{p}W\right).
\end{align*}
They are all integers, and it can be directly shown that $XW-pYZ=1$ and a matrix
$
    \begin{pmatrix}
    X & Y\sqrt p \\ Z\sqrt p & W 
    \end{pmatrix}
$ sends $z_0$ to $z$
\end{proof}

\begin{proposition}\label{prop_Ell+tr1}
Let $p$ be a prime. If $p=3$ or $p=3r+1$ for some  $r \in \Z$, then the set of elliptic points $\mathrm{Ell}^+(p)_1$ is a singleton. Otherwise, $\mathrm{Ell}^+(p)_1= \emptyset$.
\end{proposition}

\begin{proof}
Suppose there exists an elliptic point $z\in \mathrm{Ell}^+(p)_1$ and  let $\gamma=\begin{pmatrix}
a & b \sqrt p \\ c \sqrt p & -a+1
\end{pmatrix} \in R(p)^+$ be its associated elliptic element. Then we can write $z=\frac{2a-1+\sqrt{3} i}{2\sqrt p}.$ As in the proof of Proposition~\ref{prop_Ell+tr0}, we have $p\mid (a^2-a+1)$, so $p=3$ or  $p=3r+1$ for some $r \in \Z$.

Suppose $p=3$ or  $p=3r+1$, and fix an integer $a_0\in \Z$ with $a_0^2-a_0+1 \equiv 0 \pmod{p}$. Note that there are exactly two integers $n$ satisfying $n^2-n+1 \equiv 0 \pmod{p}$ in mod $p$, namely $n \equiv a_0 \pmod{p}$ or $n \equiv -a_0+1 \pmod{p}.$ For $b_0:=\frac{a_0^2-a_0+1}{p}=\frac{(-a_0+1)^2-(-a_0+1)+1}{p}$, we have two elliptic points $z_0=\frac{2a_0-1+\sqrt 3 i}{2\sqrt p}$ and $-\overline{z_0}=\frac{-2a_0+1+\sqrt 3 i}{2\sqrt p}$ that are associated to the elliptic elements $\begin{pmatrix}
a_0 & b_0\sqrt p \\ \sqrt p & -a_0+1
\end{pmatrix}$ and $\begin{pmatrix}
-a_0+1 & b_0\sqrt p \\ \sqrt p & a_0
\end{pmatrix}$, respectively. These two elliptic points are $R(p)^-$-equivalent. Indeed, as in the proof of Proposition \ref{prop_Ell+tr0}, it is enough to find integer solutions of the quadratic equation 
\begin{align*}
    \frac{a_0^2-a_0+1}{p}Z^2+(2a_0-1)ZW+pW^2=1.
\end{align*}
Since the quadratic form $\frac{a_0^2-a_0+1}{p}Z^2+(2a_0-1)ZW+pW^2$ is positive definite of discriminant~$-3$, it must have an integer solution.

It only remains to prove that any elliptic point $z \in \mathrm{Ell}^+(p)_1$ is $R(p)$-equivalent to $z_0.$ Similarly to the proof of Proposition \ref{prop_Ell+tr0}, we know that any prime factor of $c$ is 3 or of the form $3r+1$ for an integer $r \in \Z.$ Applying Lemma \ref{lem_Fermat2}, we have a pair of integers $(x_0,y_0) \in \Z^2$ such that 
\begin{enumerate}[(a)]
    \item $c=x_0^2+y_0^2$, \text{ and }
    \item $c$ divides $x_0-ay_0$ and $\frac{a-a_0}{p}x_0+\frac{a_0a+1}{p}y_0$.
\end{enumerate}
Define
\begin{align*}
    Z&:=2x_0, \quad W:=-(2a_0-1)x_0+y_0, \\
    X&:=\frac{(a_0+a-1)X+W}{c}, \quad Y:=\frac{1}{c}\left(-\frac{a_0^2-a_0+1}{p}Z+\frac{a-a_0}{p}W\right).
\end{align*}
One can see that a matrix $ \begin{pmatrix}
    X & Y\sqrt p \\ Z\sqrt p & W 
    \end{pmatrix}$ belongs to $R(p)^+$ and sends $z_0$ to $z$.
\end{proof}

\begin{proof}[Proof of Theorem~\ref{thm_ell}\normalfont{(a)}]
It follows from combining Propositions~\ref{prop_Ell+tr0} and \ref{prop_Ell+tr1}.
\end{proof}

\

Next we prove Theorem~\ref{thm_ell}(b). Suppose that $D$ is a negative integer such that $D \equiv 0,1 \pmod{4}$.
Consider the set of all (not~necessarily primitive) positive definite binary quadratic forms of discriminant $D.$ We say that two binary quadratic forms $f(X,Y)$ and $g(X,Y)$ are {\it equivalent} if there exist integers $a,b,c,d$ such that $f(x,y)=g(ax+by,cx+dy)$ and $ad-bc=\pm 1$. If $ad-bc=1$, we say that they are {\it properly equivalent}.

 Recall that any positive definite binary quadratic form of discriminant $D$ has exactly one reduced form up to equivalence. In other words, there is a bijection between the set of all equivalence classes of positive definite binary quadratic forms of discriminant $D$ and the set of all positive definite binary quadratic reduced forms of discriminant~$D.$ If we consider the set of all proper equivalence classes of primitive forms, which is a binary quadratic form whose coefficients are relatively prime, there is a natural group structure which is given via transport of structure of the ideal class group $C(D).$ We call such a group the {\it form class group} of discriminant $D$.
 
 %%Transport of structure : https://en.wikipedia.org/wiki/Transport_of_structure. Also see https://en.wikipedia.org/wiki/List_of_mathematical_jargon.

We set the following notations:
\begin{align*}
    (A,B,C):=&\text{ a binary quadratic form } AX^2+BXY+CY^2, \\
    [(A,B,C)]:=&\text{ an equivalence class of a binary quadratic form } (A,B,C), \\
    H(D):=&\text{ the set of all proper equivalence classes of positive definite binary quadratic}\\ &\text{ forms of discriminant } D.
\end{align*}

\begin{proof}[Proof of Theorem \ref{thm_ell}\normalfont{(b)}]
For a prime $p \geq 5$, any elliptic point $[z] \in \mathrm{Ell}^-(p)$ is of trace zero. We can write $z=\frac{a\sqrt p +i}{c}$, where $\begin{pmatrix}
a \sqrt p & b \\ c & -a \sqrt p
\end{pmatrix}$ is an elliptic element of $z_0$. We associate an elliptic point $z$ to a binary quadratic form $(-pb,2pa,c)=(\frac{a^2p+1}{c}p,2pa,c),$ denoted by $QF(z).$ 

We claim that the map defined by
\begin{align*}
    \mathrm{Ell}^-(p) &\longrightarrow H(-4p) \\
   [z] &\mapsto [QF(z)] 
\end{align*}
is well-defined and bijective. Indeed, let $[z_1]=[z_2] \in \mathrm{Ell}^-(p)$ be $R(p)$-equivalent elliptic elements. There exists $\gamma \in R(p)$ such that $\gamma z_1=z_2$. Moreover we may assume $\gamma:=\begin{pmatrix}
r & s\sqrt p \\ t \sqrt p & u
\end{pmatrix} \in~R(p)^+$, since if $\gamma \in R(p)^-$, then for any $\alpha \in R(p)_{z_1}^-$, we have $\gamma \alpha z_1=z_2$ and $\gamma \alpha \in R(p)^+$. It is shown by direct computation that a matrix $\begin{pmatrix}
r & s \\ tp & u 
\end{pmatrix} \in \SL_2(\Z)$ gives a proper equivalence between $QF(z_1)$ and $QF(z_2)$. In other words,
\begin{align*}
    QF(z_1)(rX+sY,tpX+uY)=QF(z_2)(X,Y).
\end{align*}

The injectivity of the map follows by reversing the above argument. If $[QF(z_1)]=[QF(z_2)]$, then there exists a matrix $\begin{pmatrix}
    r & s \\ t' & u
    \end{pmatrix} \in \SL_2(\Z)$ such that
\begin{align*}
    \begin{pmatrix}
    r & s \\ t' & u
    \end{pmatrix} 
    \begin{pmatrix}
    b_1p & a_1p \\ a_1 p & c_1
    \end{pmatrix}
    \begin{pmatrix}
    r & t' \\ s & u
    \end{pmatrix}=
    \begin{pmatrix}
    b_2p & a_2p \\ a_2 p & c_2
    \end{pmatrix}.
\end{align*}
We have $t'^2 c_1 \equiv r^2 b_1 p +4rt' a_1 p + t'^2c_1 = b_2 p \equiv 0 \pmod{p},$ so $t' \equiv 0 \pmod{p}.$ Letting $t'=pt,$  
%  the matrix$\begin{pmatrix}
%r & s \sqrt p \\ t \sqrt p & u
%\end{pmatrix} \in R(p)$
the matrix $\gamma=\begin{pmatrix}
r & s\sqrt p \\ t\sqrt p & u
\end{pmatrix}\in R(p)$ sends $z_1$ to $z_2.$

The remaining part is to prove surjectivity. Let $[(A,B',C)] \in H(-4p)$ be an equivalence class of the binary quadratic form $(A,B',C)$. Since $4$ divides its discriminant, $B'$ must be even, so we write $B'=2B.$ We consider two cases: either $p \mid A$ or not.

Suppose $A$ is divisible by $p.$ Again by the divisibility of its discriminant, $B$ is also divisible by $p$, so we write $pA$ and $pB$ instead of $A$ and $B$, respectively. Note that a quadratic form $(pA,2pB,C)$ is properly equivalent to a quadratic form $(C,-2pB,pA)$ via an equivalence $\omega_1=\begin{pmatrix}
0 & -1 \\ 1 & 0
\end{pmatrix}$, i.e.,
\begin{align*}
    (pA,2pB,C)(-Y,X)=(C,-2pB,pA)(X,Y).
\end{align*}
Let $a=-B$  and $c=A$. By the discriminant condition, we have $-b:=C=\frac{a^2p+1}{c}$. 

Next suppose $A$ is not divisible by $p$. Let $s$ be an integer such that $sA+B \equiv 0 \pmod{p}.$ Since $B^2-AC=-p$, we have $s^2 A+sB+C \equiv 0 \pmod{p}.$ In this case, we set
\begin{align*}
    a:=\frac{sA+B}{p}, \quad c:=\frac{s^2A+sB+C}{p}.
\end{align*}
Then $a$ and $c$ are integers  satisfying $-b:=\frac{a^2 p+1}{c}=A \in \Z.$ One can see that the matrix $\begin{pmatrix}
1 & s \\ 0 & 1
\end{pmatrix}$ gives a proper equivalence between two quadratic forms via
\begin{align*}
    \begin{pmatrix}
1 & 0 \\ s & 1
\end{pmatrix}
\begin{pmatrix}
A & B \\ B & C
\end{pmatrix}
\begin{pmatrix}
1 & s \\ 0 & 1
\end{pmatrix}=
\begin{pmatrix}
A & sA+B \\ sA+B & s^2A+sB+C
\end{pmatrix},
\end{align*}
so $[(A,2B,C)]=[(-b,2ap,cp)].$ 

In both cases, we have shown that the quadratic form $(A,2B,C)$ is properly equivalent to the quadratic form of the form $(-b,2ap,cp).$ Finally, we point out that there is a proper equivalence $\begin{pmatrix}
a & b \\ c & -ap
\end{pmatrix}$ between two quadratic forms $(-b,2ap,cp)$ and $(-bp,2ap,c)$. Note that a class of the quadratic form $[(-bp,2ap,c)]$ is equal to $[QF(z)]$, where $z$ is an elliptic point $z=\frac{a^2p+1}{c}.$ Therefore the map $[z] \mapsto [QF(z)]$ is surjective.
\end{proof}

\

\begin{remark}
One may slightly generalize the proof of Theorem \ref{thm_ell}(b) for an arbitrary trace $t \in \{-1,0,1\}.$ If $[z] \in \mathrm{Ell}^-(p)$ is an elliptic point and $\begin{pmatrix}
a\sqrt p & b \\ c & d\sqrt p
\end{pmatrix}$ is its associated elliptic element, then we have $z=\frac{(2a-t)\sqrt p +\sqrt{4-pt^2}i}{2c}.$ If we let $QF(z):=(-pb,p(2a-t),c)$, then the map
\begin{align*}
    \mathrm{Ell}^+(p) &\longrightarrow H(t^2p^2-4p) \\
   [z] &\mapsto [QF(z)]
\end{align*}
is also bijective. However, the set $\mathrm{Ell}^+(p)_t$ is nonempty only if $p=2$ or $3$. 
\end{remark}

\

Let $z \in \Ha$ be any elliptic point for $R(N).$ The order of $z$ is defined as
\begin{align*}
    e_z:=|R(N)_z/\{ \pm I\}|=\frac{1}{2}|R(N)_z|.
\end{align*}
It is independent of the choice of representatives of a class $[z] \in \mathrm{Ell}(N).$

\begin{proposition}
Let $z \in \Ha$ be an elliptic point for $R(p).$ Then the order of $z$ is
\begin{align*}
    e_z=\begin{cases}
        3 & \text{ if } [z] \in \mathrm{Ell}^+(p)_1, \\
        2 & \text{ otherwise.}
    \end{cases}
\end{align*}
\end{proposition}

\begin{proof}
Let $[z] \in \mathrm{Ell}^-(p)$ and $\gamma=\begin{pmatrix}
a \sqrt p & b \\ c & d\sqrt p
\end{pmatrix}$ be an elliptic element associated to $z$. One can see that the map $\alpha \mapsto \gamma\alpha$ for $\alpha \in R(p)_z^+$ induces a bijection between $R(p)_z^+/\{\pm I\}$ and $R(p)_z^-/\{\pm I\}$. Thus, in order to count the number of elements of $R(p)_z/\{\pm I\}$, it suffices to  count those of $R(p)_z^+.$ 
 By definition, the number of elements of the set $R(p)_z^+$ is equal to the number of pairs of integers $(X,Y) \in \Z^2$ such that $c$ is represented by the quadratic form associated with $z$ of  a pair $(X,Y)$, i.e., $QF(z)(X,Y)=c.$

Let $(A,2B,C)$ be a reduced form of discriminant $-4p.$ By Theorem \ref{thm_ell}(b), there is an elliptic point $[z] \in \mathrm{Ell}^-(p)$ such that $[QF(z)]=[(A,2B,C)]$. 

Suppose $A$ and $B$ are relatively prime. Then one can choose a representative $z=\frac{\hat{B}\sqrt p+1}{A},$ where $\hat{B}$ is any integer satisfying $B\hat{B} \equiv 1 \pmod{A}.$ Indeed, if we let 
\begin{align*}
    r:=-\frac{B+p\hat B}{A},
\end{align*}
then it is an integer since $B^2-AC=-p.$ Two binary quadratic forms $\left(\frac{p\hat B^2+1}{A}p,2p\hat B,A\right)$ and $(A,2B,C)$ are properly equivalent via an equivalence $\begin{pmatrix}
0 & 1 \\ -1 & r
\end{pmatrix},$ and $QF\left(\frac{\hat{B}\sqrt p+1}{A}\right)=\left(\frac{p\hat B^2+1}{A}p,2p\hat B,A\right).$
Hence we need to count the number of pairs $(X,Y)$ that represent $A$ by the binary quadratic form $(A,2B,C).$ Note that a quadratic form $(A,2B,C)$ can be written as
\begin{align}\label{eqn_qfrepreA}
    AX^2+2BXY+CY^2=A\left(X+\frac{B}{A}Y\right)^2+\left(C-\frac{B^2}{A^2}\right)Y^2.
\end{align}
Recall that $C \geq A$ since $(A,2B,C)$ is reduced. Here we assume $\gcd(A,B)=1$, so $C-~\frac{B^2}{A^2}>~A.$ Thus the quadratic form in (\ref{eqn_qfrepreA}) represents $A$ if only if $X=\pm 1$ and $Y=0$. 

Now suppose $A$ and $B$ is not relatively prime. By the discriminant condition, we have $\gcd(A,B)=p$, so let $A=pA'$ and $B=pB'.$ Again by the discriminant condition, we have $A'=\frac{pB'^2+1}{C}$ so that $(A,2B,C)=\left(\frac{pB'^2+1}{C}p,2pB',C\right)$, which is the associated quadratic form of an elliptic point $\frac{B'^2\sqrt p +1}{C} \in \Ha.$ In this case, the problem reduces to count the number of pairs $(X,Y) \in \Z^2$ for which $(A,2B,C)(X,Y)=C.$ Writing the reduced form $(A,2B,C)$ in the same way as in \eqref{eqn_qfrepreA}, we have an equation
\begin{align}\label{eqn_qfrepreC}
    CY^2+2pB'YX+pA'X^2=C\left(Y+\frac{pB'}{C}X\right)^2+\frac{p}{C}X^2=C.
\end{align}
From \eqref{eqn_qfrepreC}, $Y+\frac{pB'}{C}X$ must be $0$ or $\pm 1$. If $Y+\frac{pB'}{C}=0$ then $X^2=\frac{C^2}{p}$, which contradicts the fact that $X$ in an integer. Hence $Y+\frac{pB'}{C}=1$ or $-1$, and in both cases we have $X=0.$ Therefore for any $[z] \in \mathrm{Ell}^-(p),$ we conclude $e_z=2.$

In Proposition \ref{prop_Ell+tr0} and Proposition  \ref{prop_Ell+tr1}, we have explicitly found out what the set $\mathrm{Ell}^+(p)$ consists of. The direct computation using the similar argument as the above shows that $e_z=2$ for $z\in \mathrm{Ell}^+(p)_0$ and $e_z=3$ for $z \in \mathrm{Ell}^+(p)_1.$
\end{proof}

\

Recall that for a Fuchsian group $\Gamma$, a real number $x \in \R$ is called the {\it cusp} for $\Gamma$ if there exists a nontrivial element $\gamma \in \Gamma$ that stabilizes $x$ as a point in the complex plane. In this case, the matrix $\gamma$ is called the {\it parabolic element}. It follows that the trace of $\gamma$ is $2$ or $-2$ by definition. 

Contrary to elliptic points, there is only one cusp for any $R(p)$.

\begin{proposition}\label{prop_cusp}
Every cusp for $R(p)$ is equivalent to $\infty.$
\end{proposition}

\begin{proof}
Note that a linear fractional transformation $\begin{pmatrix}
\sqrt p & -1 \\ 1 & 0 
\end{pmatrix} \in R(p)$ sends $0$ to $\infty.$ Thus it is enough to show that any cusp for $R(p)$ is equivalent to $0.$

Let $x \in \R$ be a cusp for $R(p).$ Take its parabolic element $\gamma=\begin{pmatrix}
a & b\sqrt p \\ c \sqrt p & d
\end{pmatrix}.$ Then $x \in \R$ satisfies
\begin{align*}
    x^2-\frac{a-d}{c\sqrt p}x-\frac{b}{c}=0.
\end{align*}
Since $\mathrm{tr}~{\gamma}=2$, we have $a-d=2(a-1).$ Also, since $\det{\gamma}=1,$ we have
\begin{align*}
    -\frac{b}{c}=\frac{-pbc}{pc^2}=\frac{-ad+1}{pc^2}=\frac{(a-1)^2}{pc^2}.
\end{align*}
Hence $x$ can be written as $x=\frac{a-1}{pc}\sqrt p.$ Moreover, one can see that $(a-1)^2 = -pbc \equiv 0~\pmod{p}$, so we write $a-1=pa'$ and then $x=\frac{a'}{c}\sqrt p$. We may assume $a'$ and $c$ are relatively prime.

Suppose that $p$ and $c$ are also relatively prime. Then $\gcd(c,-pa')=1$, so there exists a matrix of the form $\begin{pmatrix}
c & -a' \sqrt p \\ *\sqrt p & *
\end{pmatrix}$ in $R(p).$ This matrix sends $x$ to $0$.

Suppose not. Then $c$ is divisible by $p$, so we write $c=pc'.$ In this case, there exists a matrix of the form $\begin{pmatrix}
c'\sqrt p & -a' \\ * & *\sqrt p
\end{pmatrix} \in R(p)$ and this matrix sends $x$ to $0.$
\end{proof}

\

\subsection{Reducing generating sets of $R(p)$ and the genera of $X(R(p))$}

We finish this section by proving that the genus $g_p$ of the curve $X(R(p))$ is zero for each of five primes $p=5,7,11,13$ and $17$. By Corollary \ref{cor_genusthm}, it suffices to show that $v_p<2\pi h_p$ in order to prove $g_p=0.$

Recall that the fundamental domain of a Fuchsian group $\Gamma$ can be given by the Dirichlet domain centered at some point $z_0 \in \Ha$, which is defined as
\begin{align*}
    D(z_0,\Gamma)=\{z \in \Ha : d_{\Ha}(z,z_0) \leq d_{\Ha}(z,\gamma z_0) \text{ for all } \gamma \in \Gamma \}.
\end{align*}

Here $z_0 \in \Ha$ is assumed to be a non-elliptic point for $\Gamma$ and $d_{\Ha}$ is the hyperbolic metric on the Poincare upper half plane $\Ha.$ One can define the Dirichlet domain of a subset $H$ of $\Gamma$ in the same manner and we denote it by $D(z_0,H)$. Obviously one has $D(z_0,\Gamma) \subseteq D(z_0,H)$ for any subset $H \subseteq \Gamma$. 

Let $\rho : \Ha \to \K$ be a conformal map which is defined by the M\"{o}bius transformation sending $z_0$ to $0$, where $\K$ be the Poincare unit disc. We can define the Dirichlet domain centered at $0$ for $\Gamma^{\rho}:=\rho\Gamma \rho^{-1}$ in $\K$ similarly. Since the conformal map $\rho$ is an isometry between $\Ha$ and $\K$, the Dirichlet domain of $\Gamma^{\rho}$ in $\K$ is equal to the image $\rho(D(z_0,\Gamma)).$ In order to distinguish two Dirichlet domains which are contained in $\Ha$ and $\K$, respectively, we denote the latter one by $D(\Gamma).$

For an element $\gamma \in \Gamma^{\rho}$, the set $I(\gamma):=\{w \in \C : d_{\K}(w,0)=d_{\K}(w, \gamma(0))\}$ is called the {\it isometric circle} of $\gamma$. Equivalently, one can define the isometric circle of $\gamma=\begin{pmatrix}
a & b \\ c & d
\end{pmatrix} \in \Gamma^{\rho}$ as $I(\gamma):=\{w \in \C : |cw+d|=1\}$. Let $\mathrm{Ext}(I(\gamma)):=\{w \in \C : |cz+d| \geq 1\}$ be the exterior domain of the circle $I(\gamma)$ with boundary. By  \cite{MR1561111}, the Dirichlet domain of~$H$ is given by the intersection of exterior domains of the isometric circles of elements in~$H$. Namely,
\begin{align*}
    D(H)=\bigcap_{\gamma \in H} \left(\K \cap \mathrm{Ext}(I(\gamma) \right). 
\end{align*}

From now on, for $\Gamma=R(p)$, we set $z_0=2i$ and $\rho(z):=\frac{z-2i}{z+2i}$.
First note that for $T_p,\omega_1$ and $s_p(n)$ given in Lemma \ref{lem_gen},
% \ww{(}
% \begin{align*}
%     H=\left\{T_p^{\rho}, \omega^{\rho}, s_p(n)^{\rho}, (T_p^{\rho})^{-1}, (\omega^{\rho})^{-1}, (s_p(n)^{\rho})^{-1} :  n \in \{\pm 1, \pm 2, \ldots, \pm \frac{p+1}{2} \}\right\},
% \end{align*}
% where \ww{)}
\begin{align*}
    &T_p^{\rho}=\begin{pmatrix}
    1+\frac{\sqrt p}{4}i & -\frac{\sqrt p}{4}i \\ \frac{\sqrt p}{4}i & 1-\frac{\sqrt p}{4}i
    \end{pmatrix}, \quad
    \omega^{\rho}= \begin{pmatrix}
    -\frac{5}{4}i & -\frac{3}{4}i \\ \frac{3}{4}i & -\frac{5}{4}
    \end{pmatrix}, \\
    &s_p(n)^{\rho}=\begin{pmatrix}
    \frac{n\hat{n}+p+1}{2\sqrt p}-(\hat n -\frac{n}{4})i & \frac{n\hat{n}-p+1}{2\sqrt p}-(\hat n +\frac{n}{4})i \\ \frac{n\hat{n}-p+1}{2\sqrt p}+(\hat n +\frac{n}{4})i & \frac{n\hat{n}+p+1}{2\sqrt p}+(\hat n -\frac{n}{4})i
    \end{pmatrix}.
\end{align*}
If we denote $I(\gamma):=\{w \in \C : |w-c|=r\}$ by $(c,r)$ then we have
\begin{align}\label{eqn_isometric-circles1}
    &I(T_p^{\rho})=\left(1+\frac{4}{\sqrt p}i, \frac{4}{\sqrt p}\right), \quad I(\omega^{\rho})=\left(\frac{5}{3}, \frac{4}{3}\right),
\end{align}
\begin{align}\label{eqn_isometric-circles2}
    I(s_p(n)^{\rho})=\left(-\frac{\frac{n\hat{n}+p+1}{2\sqrt p}+(\hat n -\frac{n}{4})i}{\frac{n\hat{n}-p+1}{2\sqrt p}+(\hat n +\frac{n}{4})i}, \frac{1}{|\frac{n\hat{n}-p+1}{2\sqrt p}+(\hat n +\frac{n}{4})i|}\right) \text{ for an integer $n$ prime to $p$},
\end{align}

\

and

\begin{align}\label{eqn_isometric-circles3}
 &I(\gamma^{-1})=\overline{I(\gamma)} \text{ for any } \gamma \in \Gamma^{\rho}.
\end{align}

Note that the generating set given in Lemma \ref{lem_gen} is not minimal. Before we calculate the isometric circles for small primes $p$, we introduce several technical lemmas and propositions that reduce the number of generators of the group $R(p)$.
%, so it is necessary to reduce the generating set to calculate isometric circles efficiently

%% Minimal generating set : https://groupprops.subwiki.org/wiki/Minimal_generating_set

Let $n_i$ for $1\leq i \leq m$ be integers relatively prime to $p$. We say that $n_1, \ldots, n_m$ are dependent (or $(n_1,\ldots,n_m)$ is dependent), if $s_p(n_i) \in \langle T_p, \omega_1, s_p(n_j) : 1 \leq j \leq m, j \neq i \rangle$ for any $1 \leq i \leq m.$
From now on, we often use the notation $s_p(n,\hat{n})$ for $s_p(n)$ given in Lemma~\ref{lem_gen} whenever it is necessary or efficient to indicate what $\hat{n}$ is. We note that for two different choices $n_1$ and $n_2$ with $n_2=n_1+pr$, we have $s_p(n,n_2)=s_p(n,n_1)\omega_1 T_p^{-r} \omega_1 ^{-1}$, so the notation $s_p(n)$ omitting  $\hat{n}$ has been so far used. 

%%문장?

\begin{lemma}\label{lem_genred1} Let $p$ be a prime.
\begin{enumerate}[\normalfont(a)]
    \item $s_p(\pm 1)$ is contained in $\langle T_p,\omega_1 \rangle.$
    \item If $n_1 \equiv n_2 \pmod{p}$, then $n_1, n_2$ are dependent.
    \item If there exists an integer $\hat{n}\in \Z$ such that $n\hat{n}+1 =p$, then $n, -\hat{n}$ are dependent.
    \item If there exists an integer $\hat{n}\in \Z$ such that $n\hat{n}+1 =-p$, then $n, \hat{n}$ are dependent.
    \item If there exists an integer $\hat{n}\in \Z$ such that $n\hat{n}+1 =p$, then $n, -n$ are dependent.
\end{enumerate}
\end{lemma}

\begin{proof}
The first and second assertion easily follows from the calculations $s_p(1,-1)=\omega_1^{-1} T_p$, $s_p(-1,1)=\omega T_p$ and $s_p(n_2)=T_p^r s_p(n_1)$, where $n_2=n_1+pr.$ To show the remaining assertions, one can check that
\begin{align*}
    \omega_1 s_p(n,\hat n)T_p \omega_1^{-1}&=s_p(-\hat n,-n-p) \quad \text{ if $n\hat n+1=p$,}\\
    \omega_1 s_p(n, \hat n) T_p \omega_1&=s_p(\hat n,n-p) \quad \text{ if $n\hat n+1=-p$,}\\
    s_p(-n,-\hat n)&=s_p(n,\hat n)^{-1} \quad \text{ if $n\hat n+1=p$.}
\end{align*}
The assertion follows from the definition of dependence.
\end{proof}

Define
\begin{align*}
    D_p^+:=\{(n_1,n_2,n_3) \in \Z^3 : &\gcd(n_i,p)=1, n_i \neq \pm 1 \text{ for any } 1 \leq i \leq 3, \\
    &\text{and there exists an integer } \hat{n_1} \in \Z \text{ such that for } \\ &a_1:=\frac{n_1 \hat{n_1}+1}{p}, \text{ }
    n_2a_1+n_1=1 \text{ and } n_3=-\hat{n_1}n_2\},
\end{align*}
\begin{align*}
    D_p^-:=\{(n_1,n_2,n_3) \in \Z^3 : &\gcd(n_i,p)=1, n_i \neq \pm 1 \text{ for any } 1 \leq i \leq 3, \\
    &\text{and there exists an integer } \hat{n_1} \in \Z \text{ such that for } \\ &a_1:=\frac{n_1 \hat{n_1}+1}{p}, \text{ }
    n_2a_1-n_1=1 \text{ and } n_3=\hat{n_1}n_2\}.
\end{align*}

\begin{lemma}\label{lem_genDp}
If $(n_1,n_2,n_3) \in D_p^+ \cup D_p^-$, then $n_1,n_2,n_3$ are dependent.
\end{lemma}

\begin{proof}
Assume $(n_1,n_2,n_3) \in D_p^+$. Let $\hat{n_2}$ be an arbitrary integer that satisfies $n_2 \hat{n_2}+1 \pmod{p}$, and set $a_2:=\frac{n_2 \hat{n_2}+1}{p}.$ Under the assumption $n_2a_1+n_1=1$, one can see that
\begin{align*}
    T_p \omega_1 s_p(n_1,\hat{n_1})s_p(n_2,\hat{n_2})=s_p(-\hat{n_1}n_2,a_1 a_2p+n_1 \hat{n_2}).
\end{align*}
Similarly, if $(n_1,n_2,n_3) \in D_p^-$, then
\begin{align*}
    T_p^{-1} \omega_1^{-1} s_p(n_1,\hat{n_1})s_p(n_2,\hat{n_2})=s_p(\hat{n_1}n_2,-a_1 a_2p-n_1 \hat{n_2}).
\end{align*}
Therefore $n_1,n_2,n_3$ are dependent.
\end{proof}

\begin{lemma}\label{lem_genred2}The following statements hold true:
\begin{enumerate}[\normalfont(a)]
    \item If $(n_1,n_2,n_3) \in D_p^+ \cup D_p^-$, then $(-n_1,-n_2,-n_3) \in D_p^+ \cup D_p^-$.
    \item If $(n_1,n_2,n_3) \in D_p^+$, then $(\hat{n_1},n_3-p,n_2') \in D_p^-$ for some integer $n_2'$ such that $n_2' \equiv n_2 \pmod{p}.$
    \item If $(n_1,n_2,n_3) \in D_p^-$, then $(\hat{n_1},n_3+p,n_2') \in D_p^-$ for some integer $n_2'$ such that $n_2' \equiv n_2 \pmod{p}.$
\end{enumerate}
Consequently, if $(n_1,n_2,n_3) \in D_p^+ \cup D_p^-$, then all of triples $(n_1,n_2,n_3)$, $(-n_1,-n_2,-n_3)$ and $(\hat{n_1},n_3,n_2)$ are dependent.
\end{lemma}

\begin{proof}
(a) It follows from the definition of $D_p^{\pm}.$

(b) Let $(m_1,m_2,m_3)=(\hat{n_1},n_3-p,n_2)$. For an integer $b_1:=\frac{m_1 \hat{m_1}+1}{p}=\frac{\hat{n_1}\hat{\hat{n_1}}+1}{p}=a_1,$ we have
\begin{align*}
    m_2b_1+m_1&=(-\hat{n_1} n_2-p)a_1+\hat{n_1} \\
    &=(-\hat{n_1}n_2 a_1-n_1\hat{n_1}-1)+\hat{n_1} \\
    &=-\hat{n_1}(n_2a_1+n_1)+\hat{n_1}-1 \\
    &=-1,
\end{align*}
and $m_3=n_2 \equiv \hat{\hat{n_1}}(-\hat{n_1}n_2) \equiv \hat{m_1}m_2 \pmod{p}.$ Thus $n_2=m_3=\hat{m_1}m_2+pr$ for some $r \in \Z.$ If we let $n_2':=\hat{m_1}m_2,$ then $n_2 \equiv n_2' \pmod{p}$ and $(m_1,m_2,n_2')=(\hat{n_1},n_3-p,n_2') \in D_p^-.$

(c) The argument of the proof is in parallel with the proof of (b).
\end{proof}

Let $n$ be an integer. We say that $n$ is {\it redundant}, if there exist integers  $n_j$ and $\epsilon, \epsilon_j \in \{ \pm 1\}$ for $1 \leq j \leq r$ such that $|n_j|<|n|$ for all $j$ and $\epsilon_1 n_1, \ldots, \epsilon_r n_r, \epsilon n$ are dependent.

\begin{proposition} For a positive integer $d$, the following two statements hold true:
\begin{enumerate}[\normalfont(a)]
    \item If $d$ divides $p-1$ and $d > \sqrt{p-1},$ then $d$ is redundant.
    \item If $d$ divides $p+1$ and $d> \sqrt{p+1}$, then $d$ is redundant.
\end{enumerate}
\end{proposition}

\begin{proof}
We only prove (a), since the proof of (b) is obtained by the same argument. Let $d'$ be an integer satisfying $dd'=p-1$. By Lemma \ref{lem_genred1}, $d$ and $d'$ are dependent and $|d'|<\sqrt{p-1}<|d|,$ so $d$ is redundant.
\end{proof}

\begin{proposition}\label{prop_p-n/d is redundant}
Let $d$ be a positive integer relatively prime to $p$ such that $d+1 <\sqrt{p},$ and let $a$ be an integer satisfying $ap \equiv 1 \pmod{d}.$
\begin{enumerate}[\normalfont(a)]
    \item If there exists $n \in \Z$ such that $an+d=1,$ then $n \equiv p \pmod{d}$ and $\frac{p-n}{d}$ is redundant.
    \item If there exists $n \in \Z$ such that $an+d=-1,$ then $n \equiv -p \pmod{d}$ and $\frac{p+n}{d}$ is redundant.
\end{enumerate}
\end{proposition}

\begin{proof}
We only prove(a). The proof of (b) follows similarly.

By the assumption of the proposition, we have
\begin{align*}
    1=an+d \equiv p^{-1}n \pmod{d},
\end{align*}
so $n \equiv p \pmod{d}.$ Since $n \leq 0,$ one has $n < p$ and $\max \{d,|n|\}=d$ or $d+1.$ Also we have $(d+1)^2<p$, hence
\begin{align*}
    (d+1)^2+n <p
\end{align*}
and so $d+1 < \frac{p-n}{d}.$ In other words, $\frac{p-n}{d}$ is the largest positive integer among the three integers $\frac{p-n}{d},d$ and $|n|.$

Let $n_1=d$, $\hat{n}_1=\frac{ap-1}{d}$ and $n_2=n.$ By Lemma \ref{lem_genDp}, $n_1,n_2,-\hat{n}_1 n_2$ are dependent and $-\hat{n}_1 n_2 \equiv -\frac{p-n}{d} \pmod{p}.$ Therefore $|-\frac{p-n}{d}|=\frac{p-n}{d}$ is redundant.
\end{proof}

We immediately get the following two corollaries by substituting $a=\pm1$ and $a=\pm 2$ in Proposition~\ref{prop_p-n/d is redundant}, respectively.

\begin{corollary}
Let $d$ be a positive integer such that $d+1 < \sqrt p.$
\begin{enumerate}[\normalfont(a)]
    \item If $d\mid p-1$, then $\frac{p-1}{d} \pm 1$ are redundant.
    \item If $d\mid p+1$, then $\frac{p+1}{d} \pm 1$ are redundant.
\end{enumerate}
\end{corollary}

\begin{corollary}
Let $d$ be a positive integer.
\begin{enumerate}[\normalfont(a)]
    \item If $d\mid p-2$ and $2 \leq d < \sqrt{p-2},$ then $\frac{p-2}{d}$ is redundant.
    \item If $d\mid p+2$ and $2 \leq d < \sqrt{p+2},$ then $\frac{p+2}{d}$ is redundant.
\end{enumerate}
\end{corollary}

We introduce a process reducing the number of generators of $R(p)$ efficiently for small primes $p$ using these lemmas and propositions.
Referring to Lemma~\ref{lem_gen}, Lemma~\ref{lem_genred1} reduces a generating set of $R(p)$ so that $R(p)=\langle T_p,\omega_1,s_p(n) : |n| \leq \frac{p-1}{2}, n\neq 0 \rangle$. Set $S=\{1,2,\ldots,\frac{p-1}{2}\}$.

\

\begin{enumerate}[\bf Step 1.]
\item Remove all divisors $d$ of $p-1$ such that $d > \sqrt{p-1}$ from $S$.
\item  Remove all divisors $d$ of $p+1$ such that $d > \sqrt{p+1}$ from $S$.
\item  Remove $\frac{p-1}{d} \pm 1$ for every divisor $d$ of $p-1$ with $d < \sqrt{p-1}$ from $S$.
\item Remove $\frac{p+1}{d} \pm 1$ for every divisor $d$ of $p+1$ with $d < \sqrt{p+1}$ from $S$.
\item  For an integer $m=\frac{p-1}{d} \pm 1$ (or $\frac{p+1}{d}\pm 1$) which is removed in \textbf{Step 3} and \textbf{Step 4}, find an integer $m' \in \{1,2,\ldots,\frac{p-1}{2}\}$ such that $mm' \equiv \pm 1 \pmod{p}.$ If $m' > d+1,$ then remove $m'$ from~$S$. 
\item Remove all divisors $d$ of $p-2$ such that $d \neq p-2$ and $d > \sqrt{p-2}$ from $S$.
\item Remove all divisors $d$ of $p+2$ such that $d \neq p+2$ and $d > \sqrt{p+2}$ from $S$.
\end{enumerate}

\

% \textbf{Step 1.} Remove all divisors $d$ of $p-1$ such that $d > \sqrt{p-1}$ from $S$.

% \bh{If you prefer the original version,  how about removing the line skips?}

% \textbf{Step 2.} Remove all divisors $d$ of $p+1$ such that $d > \sqrt{p+1}$ from $S$.

% \textbf{Step 3.} Remove $\frac{p-1}{d} \pm 1$ for every divisor $d$ of $p-1$ with $d < \sqrt{p-1}$ from $S$.

% \textbf{Step 4.} Remove $\frac{p+1}{d} \pm 1$ for every divisor $d$ of $p+1$ with $d < \sqrt{p+1}$ from $S$.
    
% \textbf{Step 5.} For an integer $m=\frac{p-1}{d} \pm 1$ (or $\frac{p+1}{d}\pm 1$) which is removed in \textbf{Step 3} and \textbf{Step 4}, find an integer $m' \in \{1,2,\ldots,\frac{p-1}{2}\}$ such that $mm' \equiv \pm 1 \pmod{p}.$ If $m' > d+1,$ then remove $m'$ from~$S$. 

% \textbf{Step 6} Remove all divisors $d$ of $p-2$ such that $d \neq p-2$ and $d > \sqrt{p-2}$ from $S$.

% \textbf{Step 7} Remove all divisors $d$ of $p+2$ such that $d \neq p+2$ and $d > \sqrt{p+2}$ from $S$.

% \

Note that \textbf{Step 5} is supported by Lemma \ref{lem_genred1}. After terminating the process, we find a set $S$ for which
\begin{align*}
    R(p) = \langle T_p, \omega_1, s_p(n), s_p(-n) : n \in S\rangle.
\end{align*}

Since $R(p)$ is {\it transpose-invariant}, i.e., for any element $A \in R(p)$, its transpose matrix~$A^t$ belongs to $R(p)$ and since $T_p^t, \omega_1^t \in \langle T_p,\omega_1\rangle$, we have
\begin{align*}
     R(p) = \langle T_p, \omega_1, s_p(n)^t, s_p(-n)^t : n \in S\rangle.
\end{align*}

Applying the above process, we reduce the set of generators for $R(p)$ for each $p\in \{5,7,11,13,17\}$ as follows:
\begin{align}\label{eqn_reduced-generating-set-of-R(p)}
\begin{cases}
    R(5)&=\langle T_5, \omega_1, s_5(2,2)^t \rangle, \\
    R(7)&=\langle T_7, \omega_1, s_7(2,3)^t \rangle, \\
    R(11)&=\langle T_{11}, \omega_1, s_{11}(2,5)^t, s_{11}(3,-4)^t \rangle, \\
    R(13)&=\langle T_{11}, \omega_1, s_{13}(2,6)^t, s_{13}(3,4)^t \rangle, \\
    R(17)&=\langle T_{17}, \omega_1, s_{17}(2,8)^t, s_{17}(3,-6)^t, s_{17}(-3,6)^t \rangle.
\end{cases}
\end{align}
We remark that $s_p(-2)^t$ is not exhibited in \eqref{eqn_reduced-generating-set-of-R(p)}, because $2$ and $-2$ are always dependent for any $p \geq 5$ by Lemma \ref{lem_genred1}(e).

\begin{proposition}
For each $p\in \{2,3,5,7,11,13, 17\}$, $X(R(p))$ has genus zero.
\end{proposition}

\begin{proof}
We already know the cases when $p=2,3$. Suppose $p \geq 5$ and let $H_p$ be the generating set of $R(p)$ given in \eqref{eqn_reduced-generating-set-of-R(p)}. Explicit formulae for isometric circles \eqref{eqn_isometric-circles1}, \eqref{eqn_isometric-circles2} and \eqref{eqn_isometric-circles3} provide the Dirichlet domains $D(H_p)$. It is immediately verified that these Dirichlet domains are hyperbolic $m_p$-gons, where $m_p=5$ for $p=5,7$, $m_p=7$ for $p=11,13$, and $m_p=9$ for $p=17.$

On the other hand, by Corollary \ref{cor_genusthm}, one can find that $h_p$ is given by
% \begin{align*}
%     h_p = \begin{cases}
%         5/2 & \text{ if } p=5, \\
%         8/3 & \text{ if } p=7, \\
%         3 & \text{ if } p=11, \\
%         19/6 & \text{ if } p=13, \\
%         7/2 & \text{ if } p=17.
%     \end{cases}
% \end{align*}
\begin{center}
\begin{tabular}{|c||c|c|c|c|c|c|}
  \hline
  $p$ & $5$ & $7$ & $11$ & $13$ &$17$ \\\hline
  $h_p$ & $5/2$ & $8/3$ & $3$ & $19/6$& $7/2$\\
  \hline
\end{tabular}\end{center}
Note that the maximum area of hyperbolic $m$-gons in $\K$ with respect to the hyperbolic measure is $(m-2)\pi.$ Moreover, if such hyperbolic $m$-gon is the Dirichlet domain $D(H_p)$, then it's area doesn't admit $(m-2)\pi$ since there exists a vertex with non-zero angle. Therefore for given $p$,
\begin{align*}
    \frac{1}{2\pi}v_p=D(R(p)^{\rho}) \leq D(H_p^{\rho}) < \frac{1}{2}m_p-1 \leq h_p.
\end{align*}
Again by Corollary \ref{cor_genusthm}, $g_p=0$ and $v_p=2\pi(h_p-2)$ for $p \in \{5,7,11,13,17\}.$
\end{proof}
We remark that the argument of the proof is not applicable to $p = 19,$ since $\frac{1}{2}m_{19}-1=\frac{7}{2} > 3=h_{19}.$

%%generator 수 좀 줄이자. 어차피 더 계산해볼 요량이니 지금 해두는 게 좋겠다.

%%표
%%그림?
%%5gon, 5gon, 5gon, ?gon.
%%h_p 계산과 결론.

\

Thompson \cite{MR0604632} proved that there are only finitely many Fuchsian groups commensurable with $\SL_2(\Z)$ of fixed genus,  up to conjugation.
For the Atkin-Lehner groups $\Gamma_0^+(p)$, it turns out that $\Gamma_0^+(p)$ has genus zero for only finitely many $p$, more precisely, for primes dividing the order of the {\it Monster group} (see \cite{MR0417184}).

It is natural to ask a similar question for the groups $R(p)$.

\begin{question}
For which prime $p$, does the modular curve $X(R(p))$ have genus $g_p=0$?
\end{question}

\section{Non-holomorphic Eisenstein series}\label{sec_Eisen}
\subsection{For the full modular group $\SL_2(\Z)$}
We start this section with the proposition which states the basic properties of non-holomorphic Eisenstein series for  $\SL_2(\Z).$
\begin{proposition}\label{prop_g}
For the non-holomorphic Eisenstein series $g(z,\overline{z},\alpha,\beta)$ as a function in $q=\alpha+\beta\in \C$ and $k=\alpha-\beta \in 2\Z$, the following holds true:
\begin{enumerate}[\normalfont(a)]
    \item The function $g(z,\z,\alpha,\beta)$ has analytic continuation in the variable $q$ to the whole plane.
    \item The function $g(z,\overline{z},\alpha,\beta)$ is a harmonic function with respect to the elliptic operator,
    $$
    \Omega_{\alpha \beta}:=-y^2\left(\frac{\partial^2}{\partial x^2}+\frac{\partial^2}{\partial y^2}\right)+i(\alpha-\beta)y\frac{\partial}{\partial x}-(\alpha+\beta)y\frac{\partial}{\partial y},
    $$
    i.e., $\Omega_{\alpha \beta}(g)=0.$
    \item For each $\gamma=\begin{pmatrix}a & b \\ c & d \end{pmatrix}  \in \SL_2(\Z)$,
    $$
    g(z,\overline{z},\alpha,\beta)|_{\alpha,\beta}\gamma := (cz+d)^{-\alpha}(c\z+d)^{-\beta}g(\gamma z,\gamma \z,\alpha,\beta)=g(z,\z,\alpha,\beta).
    $$
   % where $\gamma=\begin{pmatrix}a & b \\ c & d \end{pmatrix} \in \SL_2(\Z)$.
    \item The Fourier expansion of $g(z,\z,\alpha,\beta)$ is given by
    \begin{align*}
    g(z,\z,\alpha,\beta)=&\phi_{k/2}(y,q) \\
    &+2(-1)^{k/2}(\sqrt 2 \pi)^q \sum_{n \in \Z \setminus \{0\}} \frac{\sigma_{q-1}(n)}{\Gamma\left(\frac{q}{2}+\sgn(n)\frac{k}{2}\right)}W(2\pi n y ;\alpha,\beta)e^{2\pi i nx},
    \end{align*}
    where 
    \begin{align*}
        &\phi_{k/2}(y,q):=2\zeta(q)+(-1)^{k/2}(4\pi) 2^{1-q}\frac{\Gamma(q-1)}{\Gamma\left(\frac{q+k}{2}\right) \Gamma\left(\frac{q-k}{2}\right)}\zeta(q-1)y^{1-q}, \\
        &\sigma_{q-1}(n):=\sum_{d|n}d^{q-1}, \text{ and }\\
    &W(y;\alpha,\beta):=|y|^{-q/2}W_{\sgn(y)\frac{k}{2},\frac{1}{2}(q-1)}(2|y|)
     \text{ is the modified $W$-Whittaker function.}
     \end{align*}
    \item The completed non-holomorphic Eisenstein series $$g^*(z,\z,\alpha,\beta):=\frac{q}{2}\left(1-\frac{q}{2}\right)\pi^{-1/2}\Gamma\left(\frac{q}{2}+\frac{|k|}{2}\right)g(z,\z,\alpha,\beta)$$ satisfies the functional equations,
    \begin{align*}
        &g^*(z,\z,\overline{\beta},\alpha)=g^*(-\z,-z,\alpha,\beta), \text{ and }\\
        &g^*(z,\z,1-\alpha,1-\beta)=y^{q-1}g^*(z,\z,\beta,\alpha).
    \end{align*}
\end{enumerate}
\end{proposition}
Proposition \ref{prop_g}(b) guarantees that $g(z,\z,\alpha,\beta)$ is smooth, which is deduced from the elliptic regularity, and Proposition \ref{prop_g}(c) says that $g(z,\z,\alpha,\beta)$ is automorphic of weight $(\alpha,\beta)$ for $\SL_2(\Z)$ as we would expect. In fact, the analytic continuation of $g$ follows from the Fourier expansion of which proof is given in \cite{MR0734485}.

Obviously one can obtain the similar properties for another definition of the Eisenstein series $G_k(z,s)$ from Proposition \ref{prop_g}. In particular, let
\begin{align*}
    &\widehat{G}_k(z,s):=\pi^{-\left(s+\frac{k}{2}\right)}\Gamma \left( s+\frac{k}{2}+\frac{|k|}{2}\right) G_k(z,s), \text{ and }\\
    &\widetilde{G}_k(z,s):=\left( s+\frac{k}{2}\right) \left(s+\frac{k}{2}-1\right) \widehat G_k(z,s)
\end{align*}
be the completed and the doubly-completed Eisenstein series of $G_k(z,s)$, respectively. The Fourier expansion and the functional equations of these completed series follow from those of $g(z,\z,\alpha,\beta)$.

\begin{proposition}\label{prop_G}
The Fourier expansion of $\widehat{G}_k(z,s)$ is given by
    \begin{align*}
        \widehat{G}&_k(z,s)=\frac{\Gamma\left(s+\frac{k}{2}+\frac{|k|}{2}\right)}{\Gamma\left(s+\frac{k}{2}\right)}\hat{\zeta}(2s+k)y^s+(-1)^{k/2}\frac{\Gamma\left(s+\frac{k}{2}\right)\Gamma\left(s+\frac{k}{2}+\frac{|k|}{2}\right)}{\Gamma(s+k)\Gamma(s)}\hat{\zeta}(2-2s-k)y^{1-s-k} \\
        &+(-1)^{k/2}\Gamma\left(s+\frac{k}{2}+\frac{|k|}{2}\right)y^{-k/2}\sum_{n \in \Z \setminus \{0\}}\frac{|n|^{-s-\frac{k}{2}}\sigma_{2s+k-1}(n)}{\Gamma\left(s+\frac{k}{2}+\sgn(n)\frac{k}{2}\right)}W_{\sgn(n)k/2,s+\frac{k-1}{2}}(4\pi |n|y)e^{2\pi i nx},
    \end{align*}
    where $\hat{\zeta}(s):=\pi^{-s/2}\Gamma\left(\frac{s}{2}\right)\zeta(s)$ is the completed Riemann zeta function.
\end{proposition}

\begin{proposition}\label{prop_Gfunctional}
The completed and the doubly-completed Eisenstein series satisfy the functional equations, respectively,
    \begin{align*}
        \widehat{G}_k(z,s)=\widehat{G}_k(z,1-k-s) \quad 
        \text{ and } \quad
        \widetilde{G}_k(z,s)=\widetilde{G}_k(z,1-k-s).
    \end{align*}
\end{proposition}

The goal of this section is to establish the analogous properties for the nonholomorphic Eisenstein series for $R(p)$. In particular, it gives the Fourier expansion of the nonholomorphic Eisenstein series for the Fricke group $\Gamma_0^+(2)$ and $\Gamma_0^+(3).$

\subsection{Automorphy condition}
First of all, we prove that the Eisenstein series $g_N(z,\z,\alpha,\beta)$ and $G_{N,k}(z,s)$ satisfy the automorphy condition for $R(N).$

\begin{lemma}\label{lem_corres}
Let $N$ be a positive integer, and $a_1,b_1,c_1,d_1,a_2,b_2,c_2,d_2 \in \Z$ be any integers satisfying $a_1d_1-Nb_1c_1=1$ and $Na_2d_2-b_2c_2=1$. The following maps are bijective:
\begin{align*}
        f_{LL}:(\Z \times \Z)_{N,L} &\longrightarrow (\Z \times \Z)_{N,L} \\
        (m,n) &\mapsto (ma_1+nc_1, Nmb_1+nd_1),
    \end{align*}
\begin{align*}
        f_{RR}:(\Z \times \Z)_{N,R} &\longrightarrow (\Z \times \Z)_{N,R} \\
        (m,n) &\mapsto (ma_1+Nnc_1, mb_1+nd_1),
    \end{align*}
\begin{align*}
        f_{LR}:(\Z \times \Z)_{N,L} &\longrightarrow (\Z \times \Z)_{N,R} \\
        (m,n) &\mapsto (Nma_2+nc_2, mb_2+nd_2),
    \end{align*}
\begin{align*}
        f_{RL}:(\Z \times \Z)_{N,R} &\longrightarrow (\Z \times \Z)_{N,L} \\
        (m,n) &\mapsto (ma_2+nc_2, mb_2+Nnd_2).
    \end{align*}
\end{lemma}

\begin{proof}
We only show the lemma for the first map $f_{LL}.$ For the rest of functions the proofs can be done similarly. First we show that the map is well-defined, which means $f_{LL}(m,n) \in (\Z \times \Z)_{N,L}.$ Since $(m,n) \in (\Z \times \Z)_{N,L},$ there is a positive integer $t$ such that $(m,n) \in t\Z \times t\Z$ and $\gcd(Nm,n)=t$. It is enough to show that $\gcd(N(ma_1+nc_1),Nmb_1+nc_1)=1$, so we may assume $t=1.$ Substituting $M=Nm,$ it is reduced to show $\gcd(Ma_1+Nnc_1,Mb_1+nc_1)=1$ assuming $\gcd(M,n)=1.$ It follows from the condition $a_1d_1-Nb_1c_1=1.$

The bijectivity of $f_{LL}$ follows from the existence of the inverse map $f_{LL}^{-1}:(\Z \times \Z)_{N,L} \rightarrow (\Z \times \Z)_{N,L}$ defined by $f_{LL}^{-1}(m,n) = (d_1m-c_1n,-N n b_1+n d_1)$. In fact, this map $f_{LL}^{-1}$ is also well-defined by the same reason.
\end{proof}

\begin{proposition}\label{prop_gmodular}
Let $N$ be a positive integer and $\alpha-\beta \in 2\Z.$ For any $\gamma \in R(N),$
\begin{enumerate}[\normalfont(a)]
    \item if $\gamma \in R(N)^+,$ then 
    \begin{equation*}
        \begin{cases}
        g_{N,L}(z,\z,\alpha,\beta)|_{\alpha,\beta} \gamma=g_{N,L}(z,\z,\alpha,\beta), \\
        g_{N,R}(z,\z,\alpha,\beta)|_{\alpha,\beta} \gamma=g_{N,R}(z,\z,\alpha,\beta). 
        \end{cases}
    \end{equation*}
    \item if $\gamma \in R(N)^-,$ then
    \begin{equation*}
        \begin{cases}
        g_{N,L}(z,\z,\alpha,\beta)|_{\alpha,\beta} \gamma=g_{N,R}(z,\z,\alpha,\beta), \\
        g_{N,R}(z,\z,\alpha,\beta)|_{\alpha,\beta} \gamma=g_{N,L}(z,\z,\alpha,\beta).
        \end{cases}
    \end{equation*}
\end{enumerate}
Consequently, $g_N(z,\z,\alpha,\beta)|_{\alpha,\beta} \gamma =g_N(z,\z,\alpha,\beta)$ for any $\gamma \in \Gamma.$
\end{proposition}

\begin{proof}
Let $\gamma=\begin{pmatrix}
a & b\sqrt N \\ c\sqrt N & d
\end{pmatrix} \in R(N)^+.$ Then by Lemma \ref{lem_corres},
\begin{align*}
    g_{N,L}&(z,\z,\alpha,\beta)|_{\alpha,\beta} \gamma \\
    &=(c\sqrt N z+d)^{-\alpha}(c\sqrt N \z+d)^{-\beta}g_{N,L}(\gamma z,\gamma \z,\alpha,\beta) \\
    &=(c\sqrt N z+d)^{-\alpha}(c\sqrt N \z+d)^{-\beta}\sum_{(m,n)\in (\Z \times \Z)_{N,L}}\left(\sqrt N mz+n\right)^{-\alpha}\left(\sqrt N m\z+n\right)^{-\beta} \\
    &=\sum_{(m,n)\in (\Z \times \Z)_{N,L}}\left(\sqrt N (ma+nc)z+Nmb+nd\right)^{-\alpha}\left(\sqrt N (ma+nc)\z+Nmb+nd\right)^{-\beta} \\
    &=\sum_{f_{LL}^{-1}(m',n') \in (\Z \times \Z)_{N,L}}(m'z+n')^{-\alpha}(m'\z+n')^{-\beta} \\ 
    &=g_{N,L}(z,\z,\alpha,\beta).
\end{align*}
The others follow from Lemma~\ref{lem_corres} similarly.
\end{proof}

\

\subsection{Dual characters and the Dirichlet series}

We consider a special kind of a collection of functions $\chi : \Z \to \C$, whose images are contained in $U:=\{z \in \C : |z|=1\} \cup \{0\}.$

For a positive integer $n \in \Z_{>0}$, let $Q_n$ be a positive integer and $\chi_n : \Z \to U$ be a function such that
\begin{align*}
    \chi_n(a+Q_n)&=\chi_n(a) \quad \text{ for all $a \in \Z$}, \\
    \chi_n(a)&=0 \quad \text{ if $\gcd(n,Q_n) \neq 1$}.
\end{align*}

Let $\chi_{\bullet}$ be a sequence of pairs $(\chi_n,Q_n)$ for $n \in \Z_{>0}$. Here we call $\chi_{\bullet}$ a dual character if the following conditions hold:

\begin{enumerate}[\normalfont(i)]
    \item If $m$ and $n$ are relatively prime, then $Q_m$ and $Q_n$ are also relatively prime.
    \item If $m$ and $n$ are relatively prime, then $Q_{mn}=Q_m Q_n.$
    \item For any integers $a,b$ such that $\gcd(a,Q_m)=1$ and $\gcd(b,Q_n)=1$, $$\chi_m(a)\chi_n(b)=\chi_{mn}(aQ_n+bQ_m).$$
    \item $\chi_1=\idd_{1}$ and $Q_1=1.$
    \item There exists $\sigma \in \R$ such that $\sum_{n=1}^{\infty}\phi(Q_n)n^{-\sigma} < \infty,$ where $\phi$ is the Euler totient function.
\end{enumerate}
One of the trivial example of a dual character is $\idd_{\bullet}=(\idd_n,n)_{n \in \Z_{>0}}$, where 
\begin{align*}
    \idd_{n}(m):=\begin{cases} 1 \quad &\text{ if $m$ is relatively prime to $n$},\\
    0 \quad &\text{ otherwise},
    \end{cases} 
    % \quad \text{ and } \quad
    % Q_n=n.
\end{align*}
\begin{lemma}\label{lem_charprime}
Any dual character $\chi_{\bullet}$ is uniquely determined by its subsequence $\{(\chi_{\ell^e},Q_{\ell^e}): \ell \text{ is a prime, } e \in~\Z_{\geq 0}\}$.
\end{lemma}

\begin{proof}
Suppose that we have $Q_n$'s for $n \in \Z_{>0}$ that satisfy the conditions (i) and (ii) of the definition of a dual character. By the Chinese remainder theorem, for two integers $m,n$ that are relatively prime, a group homomorphism
\begin{align*}
    (\Z/Q_m)^{\times} \times (\Z/Q_n)^{\times} \to (\Z/Q_{mn})^{\times}
\end{align*}
which sends $(a,b)$ to $aQ_n+bQ_m$ is an isomorphism. Moreover, if $m,n$ and $r$ are pairwise relatively prime, then the following diagram commutes:
\begin{equation}\label{diag1}
\begin{tikzcd}
	{(\Z/Q_m)^{\times} \times (\Z/Q_n)^{\times} \times (\Z/Q_r)^{\times}} & {(\Z/Q_{mn})^{\times} \times (\Z/Q_r)^{\times}} & {} \\
	{(\Z/Q_m)^{\times} \times (\Z/Q_{nr})^{\times}} & {(\Z/Q_{mnr})^{\times}}
	\arrow[from=1-1, to=1-2]
	\arrow[from=1-2, to=2-2]
	\arrow[from=1-1, to=2-1]
	\arrow[from=2-1, to=2-2]
\end{tikzcd}
\end{equation}
%%diagram에 numbering하는 거 찾아보기.
The map $(\Z/Q_m)^{\times} \times (\Z/Q_n)^{\times} \times (\Z/Q_r)^{\times} \to (\Z/Q_{mnr})^{\times}$ sends $(a,b,c)$ to $aQ_{nr}+bQ_{mr}+~cQ_{mn}.$ Inductively we have an isomorphism $(\Z/Q_{m_1})^{\times} \times (\Z/Q_{m_2})^{\times} \times \cdots \times (\Z/Q_{m_j})^{\times} \to (\Z/Q_{m_1m_2 \cdots m_j})^{\times}$.

For given $\chi_{\ell^e}$, we construct $\chi_n : \Z \to \C$ for arbitrary positive integers $n\in \Z_{>0}$ as follows:
\begin{align}\label{eqn_chardecomp1}
    \chi_1(a)&:=\idd_1
\end{align}
and
\begin{align}\label{eqn_chardecomp2}
    \chi_n(a)&:=\prod_{i=1}^r \chi_{\ell_i^{e_i}}(\pi_i(a)) \quad \text{ for } n > 1,
\end{align}
where $n=\ell_1^{e_1} \cdots \ell_r^{e_r}$ is a prime factorization of $n$ and $\pi_i: (\Z/Q_n)^{\times} \to (\Z/Q_{\ell_i^{r_i}})^{\times}$ is the projection via~\eqref{diag1}. Also we define
\begin{align*}
Q_1:=1 \quad \text{ and } \quad
Q_n:=Q_{\ell_1^{e_1}}\cdots Q_{\ell_r^{e_r}} \text{ for $n>1$}.
\end{align*}
Then $\chi_{\bullet}:=((\chi_n,Q_n))_{n \in \Z_{>0}}$ is a dual character. Conversely, if one has a dual character $\chi_{\bullet}:=((\chi_n,Q_n))_{n \in \Z_{>0}}$ with given $(\chi_{\ell^e},Q_{\ell^e})$ for primes $\ell$ and non-negative integers $e \in \Z_{\geq 0}$, then it must satisfy the equations (\ref{eqn_chardecomp1}) and (\ref{eqn_chardecomp2}), which follow from the conditions (iii) and~(iv) of the definition of a dual character. Hence $\chi_{\bullet}$ is uniquely determined.
\end{proof}

For each function $\chi_n,$ the Gauss sum of $\chi_n$ is defined as
\begin{align*}
    G(b,\chi_n):=\sum_{a \in (\Z/Q_n \Z)^{\times}}\chi_n(a)(\zeta_{Q_n}^a)^b, \quad \text{ for } b \in \Z.
\end{align*}
Then we consider the associated Dirichlet series to a dual character $\chi_{\bullet}$,
\begin{align*}
    \cL(b,\chi_{\bullet},s)&:=\sum_{n=1}^{\infty}\frac{G(b,\chi_n)}{n^s}, \\
    \cL_N(b,\chi_{\bullet},s)&:=\prod_{p|N}\sum_{m=0}^{\infty}\frac{G(b,\chi_{})}{p^{ms}},\\
    \cL_p(b,\chi_{\bullet},s;v)&:=\sum_{m=0}^v \frac{G(b,\chi_{p^m})}{p^{ms}}.
\end{align*}

Note that the condition (e) of the definition of a dual character implies that there exists $\sigma \in \R$ such that the associated Dirichlet series $\cL(b,\chi_{\bullet},s)=\sum_{n=1}^{\infty}G(b,\chi_n)n^{-s}$ converges absolutely if $\Re(s) > \sigma$. The following proposition shows that  $\cL(b,\chi_{\bullet},s)$ can be expressed as the so-called Euler product.

\begin{proposition}
For any integer $b$ and a dual character $\chi_{\bullet}$, there exists $\sigma \in \R$ such that $$\cL(b,\chi_{\bullet},s)=\prod_{p}\cL_N(b,\chi_{\bullet},s).$$
for all $s \in \C$ with $\Re(s)>\sigma$. 
\end{proposition}

\begin{proof}
Let $m$ and $n$ be relatively prime. Then
\begin{align*}
    G(b,\chi_m,s)G(b,\chi_n,s)
    &=\sum_{a_1 \in (\Z/Q_m\Z)^{\times}}\chi_m(a_1)(\zeta_m^{a_1})^b\sum_{a_2 \in (\Z/Q_n\Z)^{\times}}\chi_m(a_2)(\zeta_m^{a_2})^b \\
    &=\sum_{a_1, a_2}\chi_m(a_1)\chi_n(a_2)(\zeta_n^{a_2})^b \\
    &=\sum_{a_1, a_2}\chi_{mn}(a_1Q_n+a_2Q_m)(\zeta_{mn}^{a_1Q_n+a_2Q_m})^b \\
    &=\sum_{a \in (\Z/Q_mQ_n\Z)^{\times}}\chi_{mn}(a)(\zeta_m^{a})^b.
\end{align*}
The last two equations follows from the definition of dual characters and the isomorphism given in Lemma \ref{lem_charprime}. Hence the coefficients of $\cL(b,\chi_{\bullet},s)$ are provided by a multiplicative function in $n$, so it equals to the Euler product by \cite[Theorem 1.9]{MR2378655}.
\end{proof}

One of the natural question is when the function $\cL(b,\chi_{\bullet},s)$ has analytic continuation. We suggest certain strong condition for $\chi_{\bullet}$ that admits the analytic continuation of $\cL(b,\chi_{\bullet},s)$. However, it does not characterize equivalent conditions for $\cL(b,\chi_{\bullet},s)$ to have analytic continuation precisely. %말을 잘 못 하겠는데, 'analytic continuation이 가능함'이란 조건을 대체 어떻게 써야 할까요? 문장으로 풀어쓰는 것 말고 명료한 표현이 있나요?

\begin{definition}
Let $\chi_{\bullet}$ be a dual character. We call that $\chi_{\bullet}$ is finitely generated, if there exists finitely many primes $p_1, p_2, \ldots, p_r$ for which 
\begin{align*}
    (\chi_{p_i^{e_i}},Q_{p_i^{e_i}})=(\chi_{p_i},Q_{p_i}), \quad \text{ where  } e_i \in \Z_{> 0}  \text{ for } 1\leq i \leq r, 
\end{align*}
and
\begin{align*}
    (\chi_q,Q_q)=(\idd_1,1) \quad \text{ for any prime $q \neq p_i$},\text{ } 1\leq i \leq r.
\end{align*}
The functions $\chi_{p_1^{e_1}}, \ldots, \chi_{p_r^{e_r}}$ are called the generators of $\chi_{\bullet}$.
\end{definition}

\begin{lemma}\label{lem_fgdualcha}
Let $\chi_{\bullet}$ be the finitely generated dual character with generators $\chi_{p_1}, \chi_{p_2}, \ldots, \chi_{p_r}.$ Then $$\chi_{mp_1p_2\cdots p_{r'}}=\chi_{p_1p_2\cdots p_{r'}}$$ for any positive integers $r' \leq r$ and $m \in \Z_{>0}.$ In particular, if $n$ is an integer such that $\gcd(n,p_i)\neq 1$ for $1\leq i \leq r'$ and $\gcd(n,p_j)=1$ for $r' < j \leq r$, then
$$
\chi_n=\chi_{p_1p_2\cdots p_{r'}}.
$$
\end{lemma}

\begin{proof}
First suppose a positive integer $m$ is of the form $m=p_1^{e_1}p_2^{e_2}\cdots p_{r'}^{e_{r'}}.$ Then by Lemma~\ref{lem_charprime},
\begin{align*}
    \chi_{mp_1p_2\cdots p_{r'}}(a)&=\chi_{p_1^{e_1+1}p_2^{e_2+1}\cdots p_{r'}^{e_{r'+1}}}(a)\\
    &=\chi_{p_1^{e_1+1}}(\pi_1(a))\chi_{p_2^{e_2+1}}(\pi_2(a))\cdots \chi_{p_{r'}^{e_{r'}+1}}(\pi_{r'}(a)) \\
    &=\chi_{p_1}(\pi_1(a))\chi_{p_2}(\pi_2(a))\cdots \chi_{p_{r'}}(\pi_{r'}(a)) \\
    &=\chi_{p_1p_2\cdots p_{r'}}(a)
\end{align*}
for any $a \in \Z.$ Now let $m=m_1m_2$ where $m_1$ is the largest divisor of $m$ relatively prime to $p_i$ for $1 \leq i \leq r'.$ By the same reason,
\begin{align*}
    \chi_{mp_1p_2\cdots p_{r'}}(a)&=
    \chi_{m_1}(a_1)\chi_{m_2p_1p_2\cdots p_{r'}}(a_2)\\
    &=\chi_{p_1p_2\cdots p_{r'}}(a_2)
\end{align*}
for integers $a_1,a_2$ such that $a \equiv a_1 Q_{mp_1p_2\cdots p_{r'}}+a_2Q_{m_1} \pmod{ Q_{mp_1p_2\cdots p_{r'}}}.$ Note that
\begin{align*}
    a_1 Q_{mp_1p_2\cdots p_{r'}}+a_2Q_{m_1} = a_1 Q_{p_1}Q_{p_2} \cdots Q_{p_{r'}} +a_2 \equiv a_2 \pmod{Q_{mp_1p_2\cdots p_{r'}}},
\end{align*}
so we get the desired equation.
\end{proof}

\begin{theorem}
Let $\chi_{\bullet}$ be the dual character generated by $\chi_{p_1}, \chi_{p_2}, \ldots, \chi_{p_r}.$ Then its associated Dirichlet series $\cL(b,\chi_{\bullet},s)$ is given by
\begin{align*}
    \cL(b,\chi_{\bullet},s)=\left(1-\frac{1}{p_1^s}\right)\left(1-\frac{1}{p_2^s}\right)\cdots \left(1-\frac{1}{p_r^s}\right)\zeta(s)\left(1+\sum_{\substack{1 \leq j \leq r \\ 1 \leq i_1 < i_2 < \cdots < i_j\leq r}}\frac{G(b,\chi_{p_{i_1}p_{i_2}\cdots p_{i_j}})}{(p_{i_1}^s-1)(p_{i_2}^s-1)\cdots(p_{i_j}^s-1)}\right).
\end{align*}
In particular, $\cL(b,\chi_{\bullet},s)$ is the product of meromorphic functions on $\C$, so it can be continued analytically to the complex plane $\C.$
\end{theorem}

\begin{proof}
For the sake of convenience, we abuse the notation as follows: $\frac{G(b,\chi_{p_{i_1}p_{i_2}\cdots p_{i_j}})}{(p_{i_1}^s-1)(p_{i_2}^s-1)\cdots(p_{i_j}^s-1)}:=1$ if $j=0.$
Consider a function
\begin{align*}
    f(s):=(p_1p_2 \ldots p_r)^s \pi^{-s/2}\Gamma\left(\frac{s}{2}\right)\cL(b,\chi_{\bullet},s).
\end{align*}
Then,
\begin{align*}
    f(2s)&=(p_1p_2 \ldots p_r)^{2s} \pi^{-s}\sum_{n=1}^{\infty}G(b,\chi_n)n^{-2s}\int_0^{\infty}e^{-t}t^{s-1}dt \\
    &=(p_1p_2 \ldots p_r)^{2s} \sum_{n=1}^{\infty}G(b,\chi_n)\int_0^{\infty}e^{-\pi n^2t}t^{s-1}dt \\
    &=(p_1p_2 \ldots p_r)^{2s} \sum_{\substack{0 \leq j \leq r, \\ 1\leq i_1 < i_2 < \cdots < i_j \leq r}}G(b,\chi_{p_{i_1}p_{i_2}\cdots p_{i_j}})\sum_{\substack{\gcd(n,p_t) \neq 1 \text{ for $t=i_1,i_2,\ldots,i_j$,} \\ \gcd(n,p_t)=1 \text{ for $t \neq i_1, i_2, \ldots, i_j$}}}\int_0^{\infty}e^{-\pi n^2 t}t^{s-1}dt.
\end{align*}
The last equation follows from Lemma \ref{lem_fgdualcha}.

Since the series $\sum\limits_{p_1p_2 \cdots p_k \mid n}\int_0^{\infty}e^{-\pi n^2 t}t^{s-1}dt$ converges absolutely, one can interchange the summation as

\begin{align*}
    \sum_{\substack{\gcd(n,p_t) \neq 1 \text{ for $t=i_1,i_2,\ldots,i_j$,} \\ \gcd(n,p_t)=1 \text{ for $t \neq i_1, i_2, \ldots, i_j$}}}\int_0^{\infty}e^{-\pi n^2 t}t^{s-1}dt = \sum_{\substack{j \leq k \leq r, \\ 1 \leq i_{j+1} < i_{j+2} < \cdots < i_k \leq r, \\ i_{j+1}, i_{j+2}, \ldots i_k \in \{1,2, \ldots, r\} \setminus \{i_1,i_2, \ldots, i_j\}}}\sum_{p_1p_2 \cdots p_k \mid n}\int_0^{\infty}e^{-\pi n^2 t}t^{s-1}dt.
\end{align*}
Thus,
\begin{align*}
    f(2s)&=(p_1p_2 \ldots p_r)^{2s} \times \\ &\sum_{\substack{0 \leq j \leq r, \\ 1\leq i_1 < i_2 < \cdots < i_j \leq r}}G(b,\chi_{p_{i_1}p_{i_2}\cdots p_{i_j}})\sum_{\substack{j \leq k \leq r, \\ 1 \leq i_{j+1} < i_{j+2} < \cdots < i_k \leq r, \\ i_{j+1}, i_{j+2}, \ldots i_k \in \{1,2, \ldots, r\} \setminus \{i_1,i_2, \ldots, i_j\}}}(-1)^{k-j}\sum_{p_1p_2 \cdots p_k \mid n}\int_0^{\infty}e^{-\pi n^2 t}t^{s-1}dt \\
    &=\sum_{\substack{0 \leq j \leq r, \\ 1\leq i_1 < i_2 < \cdots < i_j \leq r}}G(b,\chi_{p_{i_1}p_{i_2}\cdots p_{i_j}})\sum_{\substack{j \leq k \leq r ,\\ 1 \leq i_{j+1} < i_{j+2} < \cdots < i_k \leq r, \\ i_{j+1}, i_{j+2}, \ldots i_k \in \{1,2, \ldots, r\} \setminus \{i_1,i_2, \ldots, i_j\}}}(-1)^{k-j}\frac{(p_1p_2 \ldots p_r)^{2s}}{(p_{i_1}p_{i_2} \ldots p_{i_k})^{2s}} \\ & \qquad \times \sum_{p_1p_2 \cdots p_k \mid n}\int_0^{\infty}(p_{i_1}p_{i_2} \ldots p_{i_k})^{2s}e^{-\pi n^2 t}t^{s-1}dt.
\end{align*}
Note that the last summation in the above is
\begin{align*}
    \sum_{p_1p_2 \cdots p_k \mid n}\int_0^{\infty}(p_{i_1}p_{i_2} \ldots p_{i_k})^{2s}e^{-\pi n^2 t}t^{s-1}dt&=\sum_{n=1}^{\infty}\int_0^{\infty}e^{-\pi n^2 t}t^{s-1}dt=\pi^{-s}\Gamma(s)\zeta(2s)=\hat{\zeta}
    (2s).
\end{align*}
Also note that for arbitrary real numbers $a_1, a_2, \ldots, a_m \in \R$,
\begin{align*}
    \sum_{\substack{0 \leq j' \leq m, \\ 1\leq i_1 < i_2 < \cdots < i_{j'} \leq m}}(-1)^{m-j'}(a_{i_1}a_{i_2}\cdots a_{i_{j'}})^s=(a_1^s-1)(a_2^s-1)\cdots (a_m^s-1).
\end{align*}
By substituting the integers $p_i^2$'s except for $i=i_1, i_2, \ldots, i_j$ for $a_1,a_2, \ldots,a_m$, we have
\begin{align*}
    \sum_{\substack{j \leq k \leq r, \\ 1 \leq i_{j+1} < i_{j+2} < \cdots < i_k \leq r, \\ i_{j+1}, i_{j+2}, \ldots i_k \in \{1,2, \ldots, r\} \setminus \{i_1,i_2, \ldots, i_j\}}}(-1)^{k-j}\frac{(p_1p_2 \ldots p_r)^{2s}}{(p_{i_1}p_{i_2} \ldots p_{i_k})^{2s}}=\frac{(p_1^{2s}-1)(p_2^{2s}-1)\cdots (p_r^{2s}-1)}{(p_{i_1}^{2s}-1)(p_{i_2}^{2s}-1)\cdots (p_{i_j}^{2s}-1)}.
\end{align*}
Therefore, 
\begin{align}\label{eqn_f2s}
    f(2s)&=\hat{\zeta}(2s)\sum_{{\substack{0 \leq j \leq r, \\ 1\leq i_1 < i_2 < \cdots < i_j \leq r}}}G(b,\chi_{p_{i_1}p_{i_2}\cdots p_{i_j}})\frac{(p_1^{2s}-1)(p_2^{2s}-1)\cdots (p_r^{2s}-1)}{(p_{i_1}^{2s}-1)(p_{i_2}^{2s}-1)\cdots (p_{i_j}^{2s}-1)}.
\end{align}
The assertion follows by dividing both terms of equation (\ref{eqn_f2s}) by $(p_1p_2 \ldots p_r)^{2s} \pi^{-s}\Gamma(s).$
\end{proof}

Unfortunately, the principal character $\idd_{\bullet}$ is not finitely generated, while the Dirichlet series $\cL(b,\idd_{\bullet},s)$ associated to the principal character also has analytic continuation. 

\begin{lemma}\label{lem_trivlocalL}
The local Dirichlet series $\cL_p(b,\idd_{\bullet},s)$ satisfies the following equation:
\begin{align*}
    \cL_p(b,\idd_{\bullet},s)=
    \begin{cases}
    \zeta_p(s)^{-1} & \text{ if } p \nmid b, \\
    \zeta_p(s)^{-1}\zeta_p(s-1; v_p(b)) & \text{ otherwise},
    \end{cases}
\end{align*}
where $v_p$ denotes the $p$-valuation and $\zeta_p(s; v_p(b)):=\sum_{m=0}^{v_p(b)}\frac{1}{p^{ms}}.$
\end{lemma}
%% p^v | b인 가장 큰 v라는 말이 너무 길어보여서 그냥 p-valuation으로 줄였는데 보통 local field까지 얘기 안 할 때는 안 쓰는 말인가요? 관습적으로 그래보여서요.

\begin{proof}
It follows from the definition directly. If $b$ is prime to $p$, the Gauss sum of $\idd_{}$ for each integer $m \geq 1 $ is
\begin{align*}
    G(b,\idd_{p^m})=
    \begin{cases}
    -1 & \text{ if } m=1, \\
    0 & \text{ otherwise}.
    \end{cases}
\end{align*}
If $b$ is divisible by $p$, then
\begin{align*}
    G(b,\idd_{p^m})=
    \begin{cases}
    p^m-p^{m-1} & \text{ if } 1 \leq m \leq v_p(b), \\
    -p^{v_p(b)} & \text { if } m=v_p(b)+1, \\
    0 & \text{ otherwise}.
    \end{cases}
\end{align*}
Combining them we get the result.
\end{proof}

\begin{proposition}\label{prop_Liddformula}
The Dirichlet series $\cL(b,\idd_{\bullet},s)$ is
\begin{align*}
    \cL(b,\idd_{\bullet},s)=\frac{\sigma_{s-1}(b)}{\zeta(s)|b|^{s-1}}.
\end{align*}
Consequently, $\cL(b,\idd_{\bullet},s)$ can be continued analytically to the complex plane $\C.$
\end{proposition}

\begin{proof}
Recall that  $\cL(b,\idd_{\bullet},s)$ can be written as a product of local series $\prod_{p}\cL_p(b,\idd_{\bullet},s)$. The assertion follows by applying Lemma \ref{lem_trivlocalL}.
\end{proof}

%%다음은 L(b,1_{m\bullet},s). dual character는 아니지만 필요에 의해 나옴.

Let $\chi_{\bullet}$ be any dual character. We define another sequence of functions $\chi_{N\bullet}:=(\chi_{Nn},Q_{Nn})_{n \in \Z_{>0}}$ induced from the character $\chi_{\bullet}$. Note that $\chi_{N\bullet}$ is not a dual character in general. It is just a collection of periodic functions, but we are still able to consider its Dirichlet series $\cL(b,\chi_{N\bullet},s)$ although it is not the same as its Euler product. It is clear that the following equality holds for an arbitrary dual character~$\chi_{\bullet}$:
\begin{align}\label{eqn_localLformula}
    \cL(b,\chi_{N\bullet},s)=N^s\frac{\cL(b,\chi_{\bullet},s)}{\cL_N(b,\chi_{\bullet},s)}\prod_{p \mid N}\left( \cL_p(b,\chi_{\bullet},s)-\cL_p(b,\chi_{\bullet},s;v_p(N)-1)\right).
\end{align}
Thus if  $\cL(b,\chi_{\bullet},s)$ and its local factor $\cL_p(b,\chi_{\bullet},s)$ can be continued analytically for any prime $p$ dividing $N$, so is $\cL(b,\chi_{N\bullet},s)$. In the following subsection we fix $\chi_{\bullet}=\idd_{\bullet}$ and use this fact.

\subsection{Fourier expansion and analytic properties}

We return to the subject of the Eisenstein series $g_N(z,\z,\alpha,\beta).$ 

%W-Whittaker function과 h-function 소개.
\

Define a function
\begin{align*}
    \cg(z,\z,\alpha,\beta):=\sum_{n=\infty}^{\infty}(z+n)^{-\alpha}(\z+n)^{-\beta}
\end{align*}
for $z \in \C$ and $\alpha, \beta \in \C$ with $\Re(\alpha+\beta)>2.$ This function is 1-periodic in the variable $x=\Re(z)$, so it is given by the Fourier expansion
\begin{align}\label{eqn_cgfourier}
    \cg(z,\z,\alpha,\beta)=e^{\pi i (\beta-\alpha)/2}\sum_{n=\infty}^{\infty}h_n(y;\alpha,\beta)e^{2\pi i n x},
\end{align}
where
\begin{align*}
    h_n(y;\alpha,\beta)=e^{\pi i (\alpha-\beta)/2}\int_0^1 \cg(z,\z,\alpha,\beta)e^{-2\pi i n x}.
\end{align*}
If we define another function
\begin{align*}
    h(t;\alpha,\beta):=\int_{-\infty}^{\infty}(1-ix)^{-\alpha}(1+ix)^{-\beta}e^{-itx}dx,
\end{align*}
then the Fourier coefficients $h_n(y;\alpha,\beta)$ can be written as
\begin{align}\label{eqn_hn_and_h}
     h_n(y;\alpha,\beta)=y^{1-\alpha-\beta}h(2\pi ny;\alpha,\beta),
\end{align}
and the function $h(t;\alpha,\beta)$ satisfies the equation given in \cite[Chapter 4]{MR0734485}
\begin{align*}
    h(t;\alpha,\beta)=\frac{2\pi \cdot 2^{1-\alpha-\beta}}{\Gamma(\alpha)\Gamma(\beta)}\int_{u>|t|}e^{-u}(u+t)^{\alpha-1}(u-t)^{\beta-1}du.
\end{align*}

\

On the other hand, for $y>0$, the modified Whittaker $W$-function is defined by
\begin{align*}
    W(\epsilon y;\alpha,\beta):=y^{-q/2}W_{\frac{\epsilon k}{2},\frac{q-1}{2}}(2y),
\end{align*}
with $q=\alpha+\beta$ and $k=\alpha-\beta$. Here $W_{\kappa,\mu}$ is the standard Whittaker $W$-function which is described in \cite{MR4286926}. The integral representation of the Whittaker $W$-function is
\begin{align*}
    W_{\kappa,\mu}(y)=\frac{y^{\frac{1}{2}-\mu}e^{-\frac{y}{2}}}{\Gamma(\mu+\frac{1}{2}-\kappa)}\int_0^{\infty}e^{-u}u^{\mu-\kappa-\frac{1}{2}}(u+y)^{\mu+\kappa-\frac{1}{2}}du,
\end{align*}
and it relates the function $h(t;\alpha,\beta)$ with the modified Whittaker $W$-function as follows:
\begin{align}\label{eqn_h_and_W}
    h(t;\alpha,\beta)=\begin{cases}
    \dfrac{2\pi \cdot 2^{-q/2}|t|^{q-1}}{\Gamma(\frac{q+\sgn(t) k}{2})}W(t;\alpha,\beta) & \text{ if }t \neq 0, \\ \\
    \dfrac{2\pi \cdot 2^{1-q}\Gamma(q-1)}{\Gamma(\alpha)\Gamma(\beta)} & \text{ if }t=0.
    \end{cases}
\end{align}

% \ww{(}The following proposition shows a relation between the Dirichlet series $\cL(b,\idd_{N\bullet},s)$ and the non-holomorphic Eisenstein series $g_N(z,\z,\alpha,\beta)$, and then it implies Theorem \ref{thm_gfourier}.\ww{)}
% \bh{do we need to add "Proof of this"?}

\begin{proposition}\label{prop_gNLR}
Let $N>1$ be a positive integer. The Fourier expansions of the non-holomorphic Eisenstein series of $g_{N,L}(z,\z,\alpha,\beta)$ and $g_{N,R}(z,\z,\alpha,\beta)$ at $\infty$ are given by
\begin{align*}
    g_{N,L}(z,&\z, \alpha,\beta)=\phi_{N,L}(y,q,k) \\ 
    &+2(-1)^k\zeta(q)\left(\frac{\sqrt 2 \pi}{N}\right)^q\sum_{n\neq 0}\frac{|n|^{q-1}}{\Gamma(\frac{q}{2}+\sgn(n)k)}\cL(n,\idd_{N\bullet},s)W\left(2\pi ny/\sqrt{N};\alpha,\beta\right)e^{2 \pi i nx/\sqrt N},
\end{align*}
and
\begin{align*}
    g_{N,R}(z,&\z, \alpha,\beta)=\phi_{N,R}(y,q,k) \\ 
    &+2(-1)^k\zeta(q)\left(\frac{\sqrt 2 \pi}{\sqrt N}\right)^q\sum_{n\neq 0}\frac{|n|^{q-1}}{\Gamma(\frac{q}{2}+\sgn(n)k)}\frac{\cL(n,\idd_{\bullet},s)}{\cL_N(n,\idd_{\bullet},s)}W\left(2\pi ny/\sqrt{N};\alpha,\beta\right)e^{2 \pi i nx/\sqrt N},
\end{align*}
where $q:=\alpha+\beta$, $k:=\alpha-\beta$ and
\begin{align*}
    \phi_{N,L}(y,q,k)&:=2+2(-1)^k N^{-q}\zeta(q)h_0\left(\frac{y}{\sqrt N};\alpha,\beta\right)\cL(0,\idd_{N\bullet},s), \\
    \phi_{N,R}(y,q,k)&:=2(-1)^k \sqrt{N}^{-q}\zeta(q)h_0\left(\frac{y}{\sqrt N};\alpha,\beta\right)\frac{\cL(0,\idd_{\bullet},s)}{\cL_N(0,\idd_{\bullet},s)}.
\end{align*}
\end{proposition}

\begin{proof}
First, we note that
\begin{align*}
     g_{N,L}(z,\z, \alpha,\beta)
     &=\sum_{t=1}\sum_{\substack{(c,d) \in \Z \times \Z \\ \gcd(Nc,d)=1}} t^{-q}(\sqrt N cz+d)^{-\alpha}(\sqrt N c\z+d)^{-\beta} \\
     &=\zeta(q)\sum_{\substack{(c,d) \in \Z \times \Z \\ \gcd(Nc,d)=1}}(\sqrt N cz+d)^{-\alpha}(\sqrt N c\z+d)^{-\beta}.
\end{align*}
We let $e_{N,L}(z,\z, \alpha,\beta)=\sum\limits_{\substack{(c,d) \in \Z \times \Z \\ \gcd(Nc,d)=1}} (\sqrt N cz+d)^{-\alpha}(\sqrt N c\z+d)^{-\beta}$. Then
\begin{align*}
    e_{N,L}(z,\z, \alpha,\beta)&=2+2\sum_{\substack{c>0\\ d \in (\Z/NcZ)^{\times}}} \sum_{n \in \Z}(\sqrt N cz+d+Ncn)^{-\alpha}(\sqrt N c\z+d+Ncn)^{-\beta} \\
    &=2+2\sum_{\substack{c>0\\ d \in (\Z/NcZ)^{\times}}} \sum_{n \in \Z}(Nc)^{-q}\left(\frac{z}{\sqrt N} + \frac{d}{Nc}+n\right)^{-\alpha}\left(\frac{\z}{\sqrt N} + \frac{d}{Nc}+n\right)^{-\beta} \\
    &=2+2\sum_{\substack{c>0\\ d \in (\Z/NcZ)^{\times}}} \sum_{n \in \Z}(Nc)^{-q}\cg\left(\frac{z}{\sqrt N} + \frac{d}{Nc},\frac{\z}{\sqrt N} + \frac{d}{Nc},\alpha,\beta\right).
\end{align*}
By (\ref{eqn_cgfourier}),
\begin{align*}
    e_{N,L}(z,\z, \alpha,\beta)&=2+2\sum_{\substack{c>0\\ d \in (\Z/NcZ)^{\times}}} \sum_{n \in \Z}(Nc)^{-q}(-1)^k \sum_{n \in \Z}h_n\left(\frac{y}{\sqrt N}; \alpha, \beta\right)e^{2\pi i (\frac{x}{\sqrt N}+\frac{d}{Nc})n} \\
    &=2+2(-1)^k\sum_{n \in \Z} N^{-q}h_n\left(\frac{y}{\sqrt N}; \alpha, \beta\right) \left(\sum_{c =1}^{\infty}\sum_{d \in (\Z/Nc\Z)^{\times}}\zeta_{Nc}^{dn}c^{-q} \right)e^{2\pi i nx/\sqrt N}.
\end{align*}
As we have seen in (\ref{eqn_Gausssumprin}), the sum $\sum\limits_{\substack{0 \leq d < Nc \\ \gcd(d,Nc)=1}}\zeta_{Nc}^{dn}$ is equal to $G(n,\idd_{Nc})$, so we have
\begin{align*}
    e_{N,L}(z,\z, \alpha,\beta)=&2+2(-1)^k N^{-q}h_0\left(\frac{y}{\sqrt N}; \alpha, \beta\right)\cL(0,\idd_{N\bullet},q) \\ 
    &+2(-1)^k\sum_{n \neq 0}N^{-q}h_n\left(\frac{y}{\sqrt N}; \alpha, \beta\right)\cL(q,\idd_{N\bullet},q)e^{2\pi i nx/\sqrt N}.
\end{align*}
Applying the equations (\ref{eqn_hn_and_h}) and (\ref{eqn_h_and_W}), and multiplying by $\zeta(q)$, we get the Fourier expansion of $ g_{N,L}(z,\z, \alpha,\beta)$ given in the proposition. 

Similarly one can deduce that
\begin{align*}
    g_{N,R}(z,\z, \alpha,\beta)=2\zeta(q)\sum_{\substack{\gcd(c,N)=1 \\ d \in (\Z/c\Z)^{\times}}} \sum_{n \in \Z} (\sqrt N c)^{-q} \cg\left(\frac{z}{\sqrt N} + \frac{d}{c},\frac{\z}{\sqrt N} + \frac{d}{c},\alpha,\beta\right).
\end{align*}
The same argument as above completes the remaining part  of the proof.
\end{proof}

\begin{remark}
When $N=1$, the non-holomorphic Eisenstein series $g(z,\z,\alpha,\beta)$ is not the same as $g_1(z,\z,\alpha,\beta).$ It is $g_{1,L}(z,\z,\alpha,\beta)= g_{1,R}(z,\z,\alpha,\beta)=\frac{1}{2}g_{1}(z,\z,\alpha,\beta).$ By comparing the $n$th Fourier coefficient of $g(z,\z,\alpha,\beta)$ given in Proposition \ref{prop_g}(d) and of $g_{1,L}(z,\z,\alpha,\beta)$ given by Proposition \ref{prop_gNLR}, we have
\begin{align*}
    \zeta(s)\cL(n,\idd_{\bullet},s)\abs{n}^{s-1}=\sigma_{s-1}(n).
\end{align*}
In other words, if we denote by $\star$ the Dirichlet convolution, then
\begin{align*}
    1 \star G(n,\idd_{\bullet})(m)=\sum_{d\mid m}G(n,\idd_m)=\begin{cases}
    m & \text{ if } m\mid n, \\
    0 & \text{ otherwise. }
    \end{cases}
\end{align*}
This is the classical and famous property of the Ramanujan sum. Another proof of this formula is given in \cite[Theorem 4.1]{MR2378655}.
\end{remark}

\begin{proof}[Proof of Theorem \ref{thm_gfourier}]
It follows from Proposition \ref{prop_gNLR} by adding $g_{N,L}(z,\z,\alpha,\beta)$ and $g_{N,R}(z,\z,\alpha,\beta)$.
\end{proof}

\begin{proof}[Proof of Corollary \ref{cor_Gfourier}]
The proof is completed by applying Lemma~\ref{lem_trivlocalL} and then using the formula given in Proposition \ref{prop_Liddformula} and (\ref{eqn_localLformula}).
\end{proof}

\

For $s\in \C$, let $s^{\vee}:=1-k-s$ and
\begin{align*}
    \psi_{N,b}(s):=\frac{f_{N,b}(s^{\vee})}{f_{N,b}(s)}.
\end{align*}
Then $\psi_{N,b}(s)$ is meromorphic on $\C$ and satisfies the functional equation, \begin{align*}
    \psi_{N,b}(s)\psi_{N,b}(s^{\vee})=1.
\end{align*}

\begin{lemma}\label{lem_psi}
For a prime $p$, the function $\psi_{p,b}(s)$ is independent of the choice of $b \in \Z \setminus \{0\}$, i.e.,
\begin{align*}
    \psi_{p,b}(s)=\psi_{p,1}(s)=:\psi_p (s)
\end{align*}
for any nonzero integer $b$. 
\end{lemma}

\begin{proof}
One can verify the Lemma by direct calculations as follows: regardless of the choice of $b$,
\begin{align*}
    \psi_{p,b}(s)=\frac{(-1+p^{-s-\frac{k}{2}})\zeta_p^{-1}(2-k-2s)}{(-1+p^{s+\frac{k}{2}-1})\zeta_p^{-1}(2s+k)}=\frac{1+p^{s+\frac{k}{2}-1}}{1+p^{-s-\frac{k}{2}}}=\psi_p(s).
\end{align*}
\end{proof}

Now we are ready to prove  Theorem~\ref{thm_funeq}.
\begin{proof}[Proof of Theorem \ref{thm_funeq}.]
Denote by $CT_{p,k}(s)$ the constant term appearing in Corollary \ref{cor_Gfourier}. Write
$CT_{p,k}(s)=C_1(s)\sqrt p^{s+\frac{k}{2}}f_p(s)+C_2(s)C_3(s)f_p(s)$, where
\begin{align*}
    C_1(s)&:=\frac{\Gamma(s+\frac{k}{2}+\frac{|k|}{2})}{\Gamma(s+\frac{k}{2})}\hat{\zeta}(2s+k)y^s, \\
    C_2(s)&:=(-1)^{k/2}\frac{\Gamma(s+\frac{k}{2})\Gamma(s+\frac{k}{2}+\frac{|k|}{2})}{\Gamma(s+k)\Gamma(s)}\hat{\zeta}(2-2k-2s)y^{1-k-s}, \\
    C_3(s)&:=\sqrt p^{1-s-\frac{k}{2}}\frac{1-p^{-1}}{1-p^{-2s-k}}+\sqrt p^{-1+s+\frac{k}{2}}\frac{1-p^{1-k-2s}}{1-p^{-2s-k}}.
\end{align*}
One can observe that from the Legendre relation $2^{2s-1}\Gamma(s)\Gamma(s+\frac{1}{2})=\sqrt \pi \Gamma(2s)$, the functions $C_1(s)$ and $C_2(s)$ are interchanged with each other by dualizing $s$, i.e., $$C_1(s^{\vee})=C_2(s).$$ To prove the functional equation for $CT_{p,k}(s)$, it is enough to show that 
\begin{align}\label{eqn_C3}
    f_p(\sv)C_3(\sv)=\sqrt p^{s+\frac{k}{2}}f_p(s).
\end{align}
Note that for a prime $p$,
\begin{align*}
    C_3(s)&=\sqrt p^{1-s-\frac{k}{2}}\frac{1-p^{-1}}{1-p^{-2s-k}}+\sqrt p^{-1+s+\frac{k}{2}}\frac{1-p^{1-k-2s}}{1-p^{-2s-k}} \\
    &=\sqrt p^{1-s-\frac{k}{2}}\frac{1+p^{s+\frac{k}{2}-1}}{1-p^{-(s+\frac{k}{2})}} \\
    &=\sqrt p^{\sv+\frac{k}{2}}\psi_p(s).
\end{align*}
Multiplying by $f_p(s)$ and replacing $s$ with $\sv$, we get the equation (\ref{eqn_C3}).

For a non-zero integer $n$, the non-constant term of the Fourier series appearing in Corollary \ref{cor_Gfourier} is 
\begin{align*}
    c_n(s)=\kappa_n(s)f_{p,n}(\sv)f_p(s),
\end{align*}
where
\begin{align*}
    \kappa_n(s):=\frac{\Gamma(s+\frac{k}{2}+\frac{|k|}{2})}{\Gamma(s+\frac{k}{2}+\sgn(n)\frac{k}{2})}|n|^{-s-\frac{k}{2}}\sigma_{2s+k-1}(n)W_{\sgn(n)\frac{k}{2},s+\frac{k-1}{2}}\left(4\pi |n| \frac{y}{\sqrt p}\right)e^{2\pi i n/\sqrt p}.
\end{align*}
We claim that $c_n(\sv)=c_n(s)$, which completes the proof. One can see that $h_n(\sv)=h_n(s)$ by comparing the non-constant terms of the Fourier series of $\widehat{G}_k(z,s)$ in the functional equation given in Proposition~\ref{prop_Gfunctional}. It only remains to show
\begin{align*}
    f_{p,n}(s)f_p(\sv)=f_{p,n}(\sv)f_p(s),
\end{align*}
and this follows from Lemma \ref{lem_psi}. 
\end{proof}

It follows from Corollary~\ref{cor_Gfourier} that the completed Eisenstein series $\hat{G}_{p,k}(z,s)$ can be continued analytically to the neighborhood of $s=0$ if $k \neq 0,2$. In order to show those for $k=0,2$, we need to carry out additional calculations.

\section{Polyharmonic Maass forms for $R(p)$}\label{sec_polyhar}

%이 section은 note p.51부터 읽어나가면서 안 쓴 부분들 다 여기 넣는 거로 하자. detail 등. xi-operator의 exact sequence나 regularized inner product 등을 까먹어서 다시 다 확인해야 한다. (됐던가? 그냥 넘어갔던가? Vector-valued version을 몰라서 Rhoades의 Poincare series 논문이었나로 다른 증명 방법 적용했던 것 같다.) 
\subsection{Properties of $T_{p,m,k}(z)$}
We start this section by calculating the constant term  and the 1st Taylor coefficient of the doubly-completed Eisenstein series $\tilde{G}_{p,k}(z,s)$ for primes $p$ when $k=0,2$.

\begin{lemma}\label{lem_Tpweight0}
Let $p$ be a prime. The constant term and the 1st Taylor coefficient of the doubly-completed Eisenstein series $\tilde{G}_{p,0}(z,s)$ are given by
\begin{align*}
    T_{p,0,0}(z)&=\frac{1}{2}, \\
    T_{p,0,1}(z)&=-\frac{1}{2}(\gamma+1)+\frac{1}{2}\log (4\pi \sqrt p)-\frac{\pi}{12}\left(\sqrt p + \frac{1}{\sqrt p}\right)y + M_p(z),
\end{align*}
where $\gamma$ is the Euler–Mascheroni constant and $M_p(z)$ is a holomorphic function on $\Ha$ defined as
\begin{align*}
    M_p(z):=\sum_{n=1}^{\infty}\sigma_{-1}(n)\left(2-\frac{1-p}{1-p^{v_p(n)+1}}\right)\sin \left(\frac{2\pi n z}{\sqrt p}\right).
\end{align*}
\end{lemma}

\begin{proof}
The asymptotic formula for the Whittaker function $W_{\kappa,\mu}(z) \sim e^{-\frac{1}{2}z}z^{\kappa}$ for $\abs{\text{arg}(z)} \leq \frac{3\pi}{2}-\delta$ (for some positive $\delta$) implies that $W_{0,s-\frac{1}{2}}\left(\frac{4\pi |n|y}{\sqrt N}\right) \ll_{s} 1$. By Corollary \ref{cor_Gfourier}, for general $N=p_1^{m_1}p_2^{m_2}\cdots p_{\omega}^{m_{\omega}},$
\begin{align*}
    \widehat{G}&_{N,0}(z,s)=-\frac{1}{2s}+\hat{\zeta}(2)y\frac{(p_1-1)(p_2-1)\cdots(p_{\omega}-1)(p_1^{m_1}p_2^{m_2}\cdots p_{\omega}^{m_{\omega}}+(-1)^{\omega}p_1p_2\cdots p_{\omega})}{(2s)^{\omega}p_1^{1+\frac{m_1}{2}}p_2^{1+\frac{m_2}{2}}\cdots p_{\omega}^{1+\frac{m_{\omega}}{2}}\log p_1 \log p_2 \cdots \log p_{\omega}} \\
    &+\sum_{n \neq 0}\sigma_{-1}(n)W_{0,s-\frac{1}{2}}\left(\frac{4\pi |n|y}{\sqrt N}\right)\frac{1+\mu(N)}{(2s)^{\omega}\frac{1-p_1^{v_{p_1}(n)+1}}{1-p_1}\frac{1-p_2^{v_{p_2}(n)+1}}{1-p_2}\cdots \frac{1-p_{\omega}^{v_{p_{\omega}}(n)+1}}{1-p_{\omega}}\log p_1\log p_2\cdots \log p_{\omega}} \\
    &+(1+s^{-\omega+1})O(1)
\end{align*}
as $s$ approaches  $ 0$. In particular, if $N=p$ is a prime, we have
\begin{align*}
    \widehat{G}_{p,0}(z,s)=-\frac{1}{2s}+O(1).
\end{align*}
Therefore we get $T_{p,0,0}(z)=s(s-1)\widehat{G}_{p,0}(z,s)|_{s=0}=\frac{1}{2}.$

To prove the second formula, consider the function
\begin{align*}
    \frac{1}{2s}+CT_{p,0}(s)=\frac{1}{2s}+\hat{\zeta}(2s)p^{\frac{s}{2}}+\hat{\zeta}(2-2s)y^{1-s}\left(p^{\frac{1-s}{2}}\frac{1-p^{-1}}{1-p^{-2s}}+p^{\frac{s-1}{2}}\frac{1-p^{1-2s}}{1-p^{2s}}\right).
\end{align*}
From the asymptotic formula for the zeta function $\zeta(s)=\frac{1}{s-1}+\gamma+O(1-s)$, we have
\begin{align*}
    \hat{\zeta}(s)=\hat{\zeta}(1-s)=-\frac{1}{s}-\frac{1}{2}\log \pi +\frac{1}{2\sqrt {\pi}} \Gamma '\left(\frac{1}{2}\right)+\gamma + O(s).
\end{align*}
Here $\Gamma'(z)=\Gamma(z)\Psi(z)$, where $\Psi(z)=\frac{d}{dz}\log \Gamma(z)$ is  the digamma function. Gauss' digamma theorem gives the value $\Psi(\frac{1}{2})=-\gamma-2\log 2$, which implies $\Gamma'(\frac{1}{2})=\sqrt{\pi}(-\gamma-2\log 2)$, so we have
\begin{align*}
    \frac{1}{2s}+\hat{\zeta}(2s)p^{\frac{s}{2}}=\frac{1}{2}\gamma-\frac{1}{2}\log \sqrt p -\log 2 -\frac{1}{2}\log \pi +O(s).
\end{align*}
On the other hands,
\begin{align*}
    \lim_{s \to 0} \hat{\zeta}(2-2s)y^{1-s}\left(\frac{1-p^{-1}}{1-p^{-2s}}+p^{\frac{s-1}{2}}\frac{1-p^{1-2s}}{1-p^{2s}}\right)=\frac{\pi}{12}\left(\sqrt p + \frac{1}{\sqrt p}\right)y,
\end{align*}
so $\lim_{s \to 0}\left(\frac{1}{2s}+CT_{p,0}(s)\right)=\frac{1}{2}\gamma-\frac{1}{2}\log 4 \pi \sqrt p+\frac{\pi}{12}\left(\sqrt p + \frac{1}{\sqrt p}\right)y.$

The non-constant term of $\widehat{G}_{p,0}(z,s)$ is given by
\begin{align}\label{eqn_nonconstofGhat}
    \sum_{n \neq 0}|n|^{-s}\sigma_{2s-1}(n)W_{0,s-\frac{1}{2}}\left(4\pi |n| \frac{y}{\sqrt p}\right)f_{p,n}(1-s)f_p(s)e^{2\pi i n x/\sqrt p}.
\end{align}
For arbitrary $\delta \in \R_{>0}$, there is a constant $C>0$ such that if $y \geq \delta$, $W_{0,s-\frac{1}{2}}\left(4\pi |n| \frac{y}{\sqrt p}\right)< Ce^{-2\pi i n y/\sqrt p}$  so that the sum \eqref{eqn_nonconstofGhat} converges uniformly and absolutely as $s$ approaches $0$. Thus
\begin{align}\label{eqn_limofnonconst}
    \lim_{s \to 0} \sum_{n \neq 0}|n|^{-s}&\sigma_{2s-1}(n)W_{0,s-\frac{1}{2}}\left(4\pi |n| \frac{y}{\sqrt p}\right)f_{p,n}(1-s)f_p(s)e^{2\pi i n x/\sqrt p} \\\notag
    &=\sum_{n \neq 0}|n|^{-s} \lim_{s \to 0} \sigma_{2s-1}(n)W_{0,s-\frac{1}{2}}\left(4\pi |n| \frac{y}{\sqrt p}\right)f_{p,n}(1-s)f_p(s)e^{2\pi i n x/\sqrt p} \\
    &=  \sum_{n \neq 0}\sigma_{-1}(n)\left(1-\frac{1-p}{2(1-p^{v_p(n)+1})}\right)e^{-2\pi i n y/\sqrt p}e^{2\pi i n x/\sqrt p}=M_p(z). \notag
\end{align}
The last equality follows from that $W_{0,-\frac{1}{2}}\left(4\pi |n| \frac{y}{\sqrt p}\right)=e^{-2\pi i n y/\sqrt p}$ and $\lim_{s \to 0}f_{p,n}(1-s)f_p(s)=1-\frac{1-p}{2(1-p^{v_p(n)+1})}$. 
Therefore,
\begin{align*}
    T_{p,0,1}(z)&=-\lim_{s \to 0} \left(\hat{G}_{N,0}(z,s)+\frac{1}{2s}\right) -\frac{1}{2} \\
    &=-\frac{1}{2}(\gamma+1)+\frac{1}{2}\log (4\pi \sqrt p)-\frac{\pi}{12}\left(\sqrt p + \frac{1}{\sqrt p}\right)y + M_p(z).
\end{align*}
\end{proof}

%%결과적으로 T_{p,0,1}(z)는 polyharmonic of depth 1, 즉 harmonic maass form이다. 그런데 form을 보면 almost holomorphic modular form임을 알 수 있다. 즉, T_{p,0,1}(z)에서 +\pi/12 (\sqrt p + 1/\sqrt p)y 해준 나머지 파트가 quasi-modular form이 된다. (almost holomorphic modular form과 quasi-modular form의 correspondence). depth는 1, weight는 0. 한편으론 mock modular form이라고도 할 수 있다. (holomorphic part of harmonic maass form.) 그런데 quasi-modular form의 structure상 이는 즉 M_0(R(p)) \oplus M_{-2}(R(p))\phi, \phi:quasi-modular form of weight 2, depth 1 이라고 쓸 수 있는데. (Zagier 1-2-3 정리가 k=0에 대해서도 된다면. 서술상으로는 되는데 좀 찜찜한 게 있음.) 그런데 R(p)는 Fuchsian group of the first kind and non-cocompact니까 weight가 음수인 modular form은 없음. 즉, 이는 그냥 M_0(R(p))에 속해서 modular form이란 소리. 이건 이상하지 않나? 이상한데. SL_2(Z)에선 y가 아니라 \sqrt y 가 나와서 애초에 almost holomorphic modular form이 아니다. 그래서 오류가 안 일어나. 이는 무언가 잘못되었음을 시사.
%%아 뭐래 almost holomorphic modular form이려면 1/y에 대해 polynomial이잖아. y가 아니라. 그냥 Mock modular form이구나. quasi-modular form같은 거 아니고.

\begin{lemma}\label{lem_T2}
Let $p$ be a prime. The completed Eisenstein series $\widehat{G}_{p,2}(z,s)$ of weight $2$ is holomorphic in a neighborhood of $s=0$, and the 1st Taylor coefficient of the doubly-completed Eisenstein series $\widetilde{G}_{p,2}(z,s)$ is given by
\begin{align*}
    T_{p,2,1}(z)&=\widetilde{G}_{p,2}(z,0) \\
    &=\frac{\pi}{6}\sqrt p-\frac{p}{1+p}\frac{1}{y}-\frac{4\pi}{\sqrt p}\sum_{n=1}^{\infty}\sigma_1(n)f_{p,n}(-1)f_p(0)e^{2\pi i n z/\sqrt p}.
\end{align*}
\end{lemma}

\begin{proof}
Using the formulae $$W_{1,\frac{1}{2}}\left(4\pi |n| \frac{y}{\sqrt N}\right)=e^{-2\pi i n y/\sqrt N}4\pi |n| \frac{y}{\sqrt N} \quad \text{ and } \quad  W_{-1,\frac{1}{2}}\left(4\pi |n| \frac{y}{\sqrt N}\right)=e^{-2\pi i n y/\sqrt N}\left(4\pi |n| \frac{y}{\sqrt N}\right)^{-1},$$ one can prove the lemma similarly as in the proof of Lemma \ref{lem_Tpweight0}.
\end{proof}

%% 얘야 말로 1/y가 나오니 almost holomorphic modular form이다. 실제로 뒤에 holomorphic part(mock modular form part)가 보면...N=1일때 그냥 G_2 나오는 것처럼 생겼지. SL_2(Z)의. 그리고 1/y 파트가 quasi-modular form으로 대응시켰을 때 G2의 Q_1이고. 그러네. 보인다.

\begin{lemma}\label{lem_Geigenfunction}
For a prime $p$ and an even integer $k$, the Eisenstein series $$G_{p,k,L}(z,s), G_{p,k,R}(z,s), G_{p,k}(z,s), \widehat{G}_{p,k}(z,s) \text{ and } \widetilde{G}_{p,k}(z,s)$$ are eigenfunctions of the hyperbolic Laplacian $\Delta_k$ with the eigenvalue $s(s+k-1).$
\end{lemma}

\begin{proof}
Note that $\Delta_k=\xi_{2-k}\xi_k.$ The proof is obtained by following the arguments given in the proof of \cite[Proposition 7.1]{MR357462}.
\end{proof} 

Using Lemma \ref{lem_Geigenfunction}, we obtain the following relation which is an analogue of \cite[Proposition 8.3]{MR357462}.

\begin{proposition}\label{prop_laplacianT}
For a prime $p$ and an even  integer $k$,
\begin{align*}
    \Delta_k T_{p,n,k}(z,s)=(k-1)T_{p,n-1,k}(z,s)+T_{p,n-2,k}(z,s).
\end{align*}
\end{proposition}

\

As in the case of $N=1$, we want to show that $T_{N,m,k}(z)$ is a polyharmonic Maass form of weight $k$ and depth $\leq m+1$ when $N=p$ is a prime. Note that the polyharmonicity $\Delta_k^{m+1}T_{p,m,k}(z)=0$ and the modular invariant condition $T_{p,m,k}|_k \gamma (z)=T_{p,m,k}(z)$ for $\gamma \in R(p)$ automatically follows from Proposition \ref{prop_gmodular} and Proposition \ref{prop_laplacianT}. Thus it is enough to show that  $T_{p,m,k}(z,s)$ satisfies the moderate growth condition.

\begin{proposition}\label{prop_Tpolyhar}
The function $T_{p,m,k}(z)$ satisfies the moderate growth condition, i.e., there exists $\alpha \in \R$ such that $T_{p,m,k}(x+iy)=O(y^{\alpha})$ as $y $ approaches  $ \infty$, uniformly in $x\in \R.$ Consequently, $T_{p,m,k}(z,s)$ is a polyharmonic Maass form of weight $k$ and depth $\leq m+1$ for $R(p).$
\end{proposition}
\begin{proof}
Recall the definition of $T_{p,m,k}(z).$ We only need to look at the non-constant term of $\frac{\partial^m}{\partial s^m}|_{s=0} \widetilde{G}_{p,k}(z,s),$ which is given as follows by Corollary \ref{cor_Gfourier}:
\begin{align*}
    \sum_{n \neq 0}\sum_{j=0}^m \frac{\partial^j}{\partial s^j}\Big|_{s=0}F(n,s)\frac{\partial^{m-j}}{\partial s^{m-j}}\Big|_{s=0}\left(y^{-k/2}W_{\sgn(n)\frac{k}{2},s+\frac{k-1}{2}}(4\pi |n|y/\sqrt p) e^{2\pi i nx/\sqrt p}\right).
\end{align*}
Here, $$F(n,s):=(-1)^{k/2}\left(s+\frac{k}{2}\right)\left(s+\frac{k}{2}-1\right)\frac{\Gamma(s+\frac{k}{2}+\frac{|k|}{2})}{\Gamma(s+\frac{k}{2}+\sgn(n)\frac{k}{2})}|n|^{-s-\frac{k}{2}}\sigma_{2s+k-1}(n) f_{p,n}(1-k-s)f_p(s).$$ The part $\frac{\partial^{m-j}}{\partial s^{m-j}}|_{s=0}\left(y^{-k/2}W_{\sgn(n)\frac{k}{2},s+\frac{k-1}{2}}(4\pi |n|y/\sqrt p) e^{2\pi i nx/\sqrt p}\right)$ decays exponentially as $y$ approaches $  \infty$ by \cite[Corollary A.3]{MR3856183}. Thus it suffices to show that for any $\epsilon >0$,
\begin{align*}
    \frac{\partial^j}{\partial s^j}|_{s=0}F(n,s) \ll e^{\epsilon |n|} \qquad \text{as } |n| \to \infty,
\end{align*}
for $0 \leq j \leq m.$ Let $F_0(s):=(-1)^{k/2}\left(s+\frac{k}{2}\right)\left(s+\frac{k}{2}-1\right)\frac{\Gamma(s+\frac{k}{2}+\frac{|k|}{2})}{\Gamma(s+\frac{k}{2}+\sgn(n)\frac{k}{2})}.$ Then
\begin{align*}
    \frac{\partial^j}{\partial s^j}\Big|_{s=0}F(n,s)=\sum_{\substack{j_1+j_2+j_3=j \\ j_1,j_2,j_3 \geq 0}} \frac{\partial^{j_1}}{\partial s^{j_1}}\Big|_{s=0}F_0(s) \frac{\partial^{j_2}}{\partial s^{j_2}}\Big|_{s=0} |n|^{-s-\frac{k}{2}}\sigma_{2s+k-1}(n)  \frac{\partial^{j_3}}{\partial s^{j_3}}\Big|_{s=0} f_{p,n}(1-k-s)f_p(s).
\end{align*}
The first term $\frac{\partial^{j_1}}{\partial s^{j_1}}|_{s=0}F_0(s)$ does not depend on $n$. For the second term, note that
\begin{align*}
    \frac{\partial^{j}}{\partial s^{j}}\Big|_{s=0}\sigma_{2s+k-1}(n)=\sum_{d \mid |n|}2^j (\log d)^j d^{2s+k-1}.
\end{align*}
Hence we have
\begin{align*}
    \frac{\partial^{j_2}}{\partial s^{j_2}}\Big|_{s=0} |n|^{-s-\frac{k}{2}}\sigma_{2s+k-1}(n) \ll |n|^{\frac{3}{2}k-1}(\log |n|)^j \qquad \text{ as } |n| \to \infty.
\end{align*}
Lastly, in the third term $f_{p,n}(1-k-s)f_p(s)=f_{p,n}(s)f_p(1-k-s)$ by Lemma \ref{lem_psi}. Note that $v_p(n) \ll \log |n|$. By the definition of $f_{p,n}(s)$, we get that
\begin{align*}
    f_{p,n}(s) \leq  \frac{p^{-(s-1)-\frac{k}{2}}\left(\left(\sum_{i=1}^{v_p(n)}(p-1)p^{i-1+i(2s+k-2)}\right)-p^{v_p(n)+(v_p(n)+1)(2s+k-2)}\right)+1}{\zeta_p^{-1}(-2s-k+2)},
\end{align*}
so
\begin{align*}
    \frac{\partial^j f_{p,n}(s)}{\partial s^j}\Big|_{s=0} \ll_{p,k,j} (\log |n|)^{2j}.
\end{align*}
Since $f_p(s):=f_{p,1}(s)$, this proves that the third term $\frac{\partial^{j_3}}{\partial s^{j_3}}|_{s=0} f_{p,n}(1-k-s)f_p(s)$ grows at most logarithmically, so $\frac{\partial^j}{\partial s^j}|_{s=0}F(n,s) \ll e^{\epsilon |n|}$ for any $\epsilon >0$ as we have desired.
\end{proof}

\begin{remark}
By Proposition \ref{prop_Tpolyhar}, $T_{p,2,1}(z)$ is a polyharmonic Maass form of weight 2 and depth $\leq 1$ for $R(p).$ In Lemma \ref{lem_T2} we show that $T_{p,2,1}(z)$ is of the form $t_0(z)+t_1(z)y^{-1}$, where $t_0(z)$ and $t_1(z)$ are holomorphic functions. Thus $T_{p,2,1}(z)$ is in fact an almost holomorphic modular form of weight 2 for $R(p).$ The function $t_0(z)=\frac{\pi}{6}\sqrt p-\frac{4\pi}{\sqrt p}\sum_{n=1}^{\infty}\sigma_1(n)f_{p,n}(-1)f_p(0)e^{2\pi i n z/\sqrt p}$ is a holomorphic part of a harmonic Maass form, namely, a mock modular form. At the same time, it is a constant term of some almost holomorphic modular form in a variable $y^{-1},$ so it is a quasi-modular form of weight 2 and depth 1 for $R(p).$ For any cocompact Fuchsian group there exists essentially one quasi-modular form of weight 2, and that form and the ring of modular forms generates the ring of quasi-modular forms for such a group. In this case, $t_0(z)$ does play that role. 
\end{remark}

By following the similar argument as in the proof of \cite[Theorem 8.4]{MR357462}, one can prove that if $k\neq 2$, then $\Delta_k^m T_{p,m,k}(z) \neq 0$. Similarly it can be also proved that $\Delta_k^{m-1} T_{p,m,2}(z) \neq 0$. Thus we obtain the following proposition.

\begin{proposition}\label{prop_T}
Let $p$ be a prime and $k \neq 2$ be an even integer. 
\begin{enumerate}[\normalfont(a)]
    \item The function $T_{p,k,m}(z)$ is a polyharmonic Maass form of depth $m+1$ for $R(p)$ with $\Delta_k^{m}T_{p,2,m}(z) \neq 0$.
    \item The function $T_{p,2,m}(z)$ is a polyharmonic Maass form of depth $m$ for $R(p)$ with $\Delta_k^{m-1}T_{p,2,m}(z) \neq 0$.
\end{enumerate}
\end{proposition}

\

We summarize as follows.
\begin{proposition}
Let $E_k^m(R(p))$ be the $\C$-vector space spanned by the Taylor coefficients $T_{p,n,k}(z)$ for $n \leq m$ of the Eisenstein series $\widetilde{G}_{p,k}(z,s).$ Then $E_k^m(R(p))$ is a $m$-dimensional subspace of $V_k^m(R(p))$ and its basis is given by
\begin{align*}
    \begin{cases}
    \{T_{p,0,k}(z), T_{p,1,k}(z), \ldots, T_{p,m-1,k}(z)\} & \text{ if } k \neq 2, \\
    \{T_{p,1,2}(z), T_{p,2,2}(z), \ldots, T_{p,m,2}(z)\} & \text{ if } k=2.
    \end{cases}
\end{align*}
\end{proposition}

\

\subsection{Fourier expansions of polyharmonic Maass forms}

In parallel with the case of $\SL_2(\Z)$, we find the general form of the Fourier expansion of a polyharmonic Maass form for $R(N).$ For $s_0  \in \C$ such that $s_0 \neq \frac{1-k}{2}$ and $m \geq 0$, let
\begin{align*}
\begin{cases}
    u_{k,0}^{[m],+}(y; s_0):=(\log y)^m y^{s_0}, \\
    u_{k,0}^{[m],-}(y; s_0):=(-1)^m(\log y)^m y^{1-k-s_0}, \\
\end{cases}
\end{align*}
and for $\epsilon \in \{\pm 1\}$, let
\begin{align*}
    u_{N,\epsilon,k,|n|}^{[m],-}(y;s_0):=y^{-\frac{k}{2}}\frac{\partial^m}{\partial s^m}\Big|_{s=s_0}W_{\frac{1}{2}\epsilon k,s+\frac{k-1}{2}}\left(4\pi |n| \frac{y}{\sqrt N}\right).
\end{align*}
\begin{proposition}
Let $N>1$ be an integer and $f(z)$ be a shifted polyharmonic Maass form of weight $k$ and depth $m$ for $R(p)$, whose eigenvalue is $c \in \C$. Let $s_0 \in \C$ satisfy $s_0(s_0+k-1)=c.$ Then the Fourier expansion of $f(z)$ is given by
\begin{align*}
    f(z)=\sum_{j=0}^{m-1}\left(c_{0,j}^+  u_{k,0}^{[j],+}(y; s_0)+c_{0,j}^-  u_{k,0}^{[j],-}(y; s_0)\right) + \sum_{\epsilon \in \{\pm 1\}}\sum_{n=1}^{\infty}\sum_{j=0}^{m-1} c_{\epsilon,j,n}^-  u_{N,\epsilon,k,|n|}^{[j],-}(y;s_0) e^{2\pi i \epsilon n x/\sqrt N},
\end{align*}
where $c_{0,j}^+, c_{0,j}^-, c_{\epsilon,j,n}^- \in \C.$ In particular, if $f(z)$ is a harmonic Maass form of weight $k$ for $R(p)$, then the Fourier expansion of $f$ is given by
\begin{align*}
    f(z)=\sum_{n=1}^{\infty}b_{-n}\Gamma\left(1-k,4\pi |n|\frac{y}{\sqrt N}\right) e^{-2\pi i nz/\sqrt N}+(b_0y^{1-k}+a_0) + \sum_{n=1}^{\infty}a_n e^{2\pi i n z/\sqrt N},
\end{align*}
for $b_{-n},a_n \in \C$.
\end{proposition}
\begin{proof}
The proof is completely in parallel with one for $\SL_2(\Z)$ given in \cite[Theorem 4.3]{MR357462} and \cite[Lemma 4.4]{MR357462}.
\end{proof}

\

\subsection{Structure of polyharmonic Maass forms}

Propositions in this subsection can be proved by following the same line in \cite{MR357462}, once we prove the analytic properties of the non-holomorphic Eisenstein series. There is no new idea of the proof, so we state our results briefly.

Recall that for any Fuchsian group $\Gamma$ for which $X(\Gamma)$ has genus $g$, the dimension of the space of weight $2$ modular forms for $\Gamma$ is $g+t-1$, where $t$ is the number of cusps of $\Gamma$ (the proof is given in \cite{MR1021004}). In particular, if $g_p=0$ for a prime $p$, then there is no non-zero weight $2$ modular form for $R(p).$ This observation implies the following propositions.

\begin{proposition}
Let $p$ be a prime for which the modular curve $X(R(p))$ has genus $g_p=0.$
\begin{enumerate}[\normalfont(a)]
    \item The space of weight $0$ modular forms $V_0^{1/2}(R(p))$ is $1$-dimensional, that is, $V_0^{1/2}(R(p))=\C.$
    \item For any integer $m \geq 0$, $V_0^{m+\frac{1}{2}}(R(p))=V_0^{m+1}(R(p)).$
    \item For any integer $m \geq 0$, $\dim V_0^{m+1}(R(p)) \leq \dim V_0^{m}(R(p))+1.$ Consequently, $\dim V_0^m(R(p)) \leq m$ for $m \geq 1.$
    \item The space of weight $2$ harmonic Maass forms $V_2^1(R(p))$ is $1$-dimensional, that is, $V_2^1(R(p))=\C T_{p,1,2}(z).$
    \item For any integer $m \geq 0$, $V_2^{m}(R(p))=V_2^{m+\frac{1}{2}}(R(p)).$
    \item For any integer $m \geq 0$, $\dim V_2^{m+1}(R(p)) \leq \dim V_2^{m}(R(p))+1.$ Consequently, $\dim V_2^m(R(p)) \leq m$ for $m \geq 1.$
\end{enumerate}
\end{proposition}

% Let $p$ be a prime and $V_k^m(p)$ be the space of polyharmonic Maass forms for $R(p)$ with eigenvalue $0$. Lagarias and Rhoades \cite{MR357462} found a modified basis of the space $V_k^m$ that satisfies certain `switching' conditions for the full modular group $\SL_2(\Z)$. In the same way, we establish the specific examples of  polyharmonic Maass forms for $R(p)$, that is, the Taylor coefficients of the non-holomorphic Eisenstein series. 
% It turns out that the doubly-completed Eisenstein series $\widetilde{G}_{p,k}(z,s)$ can be continued holomorphically in a neighborhood of $s=0$ (see Section \ref{}). Denote the nth  Taylor coefficient of $\widetilde{G}_{p,k}(z,s)$ in $s \in \C$ near $s=0$ by $T_{p,k,n}(z).$ Namely, the doubly-completed Eisenstein series $\widetilde{G}_{p,k}(z,s)$ is written as
% \begin{align*}
%     \widetilde{G}_{p,k}(z,s)=\sum_{n=0}^{\infty}T_{p,k,n}(z)s^n,
% \end{align*}
% whenever $s$ is contained in a small neighborhood of $0$.

% Once we prove that the function $T_{p,n,k}(z)$ is a polyharmonic Maass form, it follows from the same argument as in \cite{MR357462} that the functions $T_{p,n,k}(z)$ together with the cusp forms span all the space of polyharmonic Maass forms for $R(p)$.

\begin{proposition}\label{prop_polyhar}
Let $k \geq 4$ be an even integer and $m \geq 0$ be an integer.
\begin{enumerate}[\normalfont(a)]
    \item Any $1$-harmonic Maass form of weight $k$ for $R(p)$ is holomorphic, i.e., $V_k^1(R(p))=V_k^{1/2}(R(p))=M_k(R(p))$.
    \item The space $V_{2-k}^1(R(p))$ is spanned by $T_{p,2-k,0}(z)=\widetilde{G}_{p,2-k}(z,0)$.
    \item For $m \geq 1$, $V_k^m(R(p))=E_k^m(R(p))+S_k(R(p))$, where $E_k^m(R(p))$ is the $m$-dimensional so-called Eisenstein space spanned by $T_{p,k,0}(z), \cdots , T_{p,k,m-1}(z)$, and $S_k(R(p))$ is the space of cusp forms of weight $k$ for $R(p)$. Moreover, $V_k^m(R(p))=V_k^{m-\frac{1}{2}}(R(p)).$ 
    \item For $m \geq 0$, $V_{2-k}^m(R(p))=E_{2-k}^m(R(p))$ which is the $m$-dimensional Eisenstein space spanned by $T_{p,2-k,0}(z), \cdots , T_{p,2-k,m-1}(z)$. Moreover, $V_{2-k}^m(R(p))=V_{2-k}^{m+\frac{1}{2}}(R(p)).$
\end{enumerate}
\end{proposition}

Note that for any Fuchsian group of the first kind, the cusp form of negative weight is always $0$. Hence for $k \neq 0,2$ we summarize Proposition~\ref{prop_polyhar} as
\begin{itemize}
    \item $V_k^m(p)=E_k^m(p)\oplus S_k(p)=\langle T_{p,k,0}(z), \cdots , T_{p,k,m-1}(z) \rangle \oplus S_k(p)$,
 \item $V_k^{m -\sgn(k)\frac{1}{2}}(p)=V_k^m(p)$.
\end{itemize}

%%sturm bound 까먹지 말기.
%%까먹기. 누가 이미 썼다.

\def\cprime{$'$} \def\Dbar{\leavevmode\lower.6ex\hbox to 0pt{\hskip-.23ex
  \accent"16\hss}D} \def\cftil#1{\ifmmode\setbox7\hbox{$\accent"5E#1$}\else
  \setbox7\hbox{\accent"5E#1}\penalty 10000\relax\fi\raise 1\ht7
  \hbox{\lower1.15ex\hbox to 1\wd7{\hss\accent"7E\hss}}\penalty 10000
  \hskip-1\wd7\penalty 10000\box7}
  \def\polhk#1{\setbox0=\hbox{#1}{\ooalign{\hidewidth
  \lower1.5ex\hbox{`}\hidewidth\crcr\unhbox0}}} \def\dbar{\leavevmode\hbox to
  0pt{\hskip.2ex \accent"16\hss}d}
  \def\cfac#1{\ifmmode\setbox7\hbox{$\accent"5E#1$}\else
  \setbox7\hbox{\accent"5E#1}\penalty 10000\relax\fi\raise 1\ht7
  \hbox{\lower1.15ex\hbox to 1\wd7{\hss\accent"13\hss}}\penalty 10000
  \hskip-1\wd7\penalty 10000\box7}
  \def\ocirc#1{\ifmmode\setbox0=\hbox{$#1$}\dimen0=\ht0 \advance\dimen0
  by1pt\rlap{\hbox to\wd0{\hss\raise\dimen0
  \hbox{\hskip.2em$\scriptscriptstyle\circ$}\hss}}#1\else {\accent"17 #1}\fi}
  \def\bud{$''$} \def\cfudot#1{\ifmmode\setbox7\hbox{$\accent"5E#1$}\else
  \setbox7\hbox{\accent"5E#1}\penalty 10000\relax\fi\raise 1\ht7
  \hbox{\raise.1ex\hbox to 1\wd7{\hss.\hss}}\penalty 10000 \hskip-1\wd7\penalty
  10000\box7} \def\lfhook#1{\setbox0=\hbox{#1}{\ooalign{\hidewidth
  \lower1.5ex\hbox{'}\hidewidth\crcr\unhbox0}}}
\providecommand{\bysame}{\leavevmode\hbox to3em{\hrulefill}\thinspace}
\providecommand{\MR}{\relax\ifhmode\unskip\space\fi MR }
\providecommand{\MRhref}[2]{%
  \href{http://www.ams.org/mathscinet-getitem?mr=#1}{#2}
}
\providecommand{\href}[2]{#2}

\end{document}